\documentclass[10pt]{article}
\usepackage[latin1]{inputenc}
\usepackage{epsfig}
\usepackage{color}
\usepackage[british,english]{babel}
\usepackage{amsthm}
\usepackage{amsmath}
\usepackage{amsfonts}
\usepackage{amssymb}
\usepackage{graphicx}
\usepackage{fancyhdr}

\newtheorem{corollary}{Corollary}[section]

\newtheorem{lemma}[corollary]{Lemma}

\newtheorem{remark}[corollary]{Remark}
\newtheorem{theorem}[corollary]{Theorem}

\date{}
\begin{document}
\title{Homogenization of  Bingham Flow in  thin porous media}\maketitle

\vskip-30pt
 \centerline{Mar\'ia ANGUIANO\footnote{Departamento de An\'alisis Matem\'atico. Facultad de Matem\'aticas.
Universidad de Sevilla, 41012 Sevilla (Spain)
anguiano@us.es} and Renata BUNOIU\footnote{Université de Lorraine, CNRS, IECL, 57000 Metz (France)
renata.bunoiu@univ-lorraine.fr}}


 \renewcommand{\abstractname} {\bf Abstract}
\begin{abstract} 
By using dimension reduction and homogenization techniques, 
we study the steady flow of an incompresible   viscoplastic Bingham fluid in a  thin porous medium.
A main feature of our study is the dependence of  the yield stress
of the Bingham fluid on the  small parameters  describing the geometry of the 
thin porous medium under consideration.
Three different problems are obtained in the limit when the small parameter $\varepsilon$ tends to zero, following the ratio
between the height $\varepsilon$ of the porous medium and the relative dimension $a_\varepsilon$ of its periodically distributed 
pores. We conclude with  the interpretation of these limit problems, which all preserve the nonlinear character of the flow.
\end{abstract}

 {\small \bf AMS classification numbers:} 76A05, 76A20, 76M50, 35B27.  
 

\section{Introduction}

We study in this paper the steady incompressible flow of a Bingham fluid in a thin porous medium
containing an array of vertical cylindrical obstacles (the pores). 
The model of thin porous medium of thickness much smaller than the distance between the pores
was introduced in \cite{Zhengan}, where a stationary incompressible Navier-Stokes flow was studied.
Recently, the model of thin porous medium under consideration in this paper was introduced
in \cite{Fabricius}, where the flow of an incompressible viscous fluid described by the stationary
Navier-Stokes equations was studied by the multiscale expansion method, which is a formal but 
powerful tool to analyse homogenization problems. These results were rigorously proved 
in \cite{Anguiano_SuarezGrau2} using an adaptation introduced in \cite{Anguiano_SuarezGrau} of the 
periodic unfolding method 
 from \cite{CDG}. This adaptation consists of a combination of the unfolding
method with a rescaling in the height variable, in order to work with a domain of fixed height,
and to use monotonicity arguments to pass to the limit. In \cite{Anguiano_SuarezGrau}, in particular,
the flow of an incompressible stationary Stokes system with a nonlinear viscosity, being a power law, 
was studied. For non-stationary incompressible viscous flow in a thin porous medium see \cite{Anguiano},
where a non-stationary Stokes system is considered, and \cite{Anguiano2}, where a non-stationary non-newtonian 
Stokes system, where the viscosity obeyed the power law, is studied.
For the periodic unfolding method applied to the study
of problems stated in other type of thin periodic domains we refer for instance to
\cite{Griso} for crane type structures and to \cite{Griso-Mer}, \cite{Griso-Orlik} for thin layers with thin beams
structures, where elasticity problems are studied.

If $\Pi$ is a three-dimensional domain with smooth boundary $\partial \Pi$ and  $f =(f_1,f_2,f_3)$ 
are external given forces defined on 
 $\Pi$, then the velocity $u=(u_1,u_2,u_3)$ of a fluid and its pressure 
 $p$ satisfy the  equations of motion
 \begin{equation}\label{motion}
   - \sum _{j=1}^{3} \partial_{x_i} (\sigma(p,u))_{ij}
= f_{ i}\, \, \, \hbox {in}\,\,  \Pi, \quad 1 \leq i \leq 3,
\end{equation}
completed with the fluid's incompressibility condition 
$ {\textrm{div}}\, u =\sum _{i=1}^{3}\partial_{x_i}{u_i}=0 \hbox { in }
 \Pi$,
and the  no-slip boundary condition 
$  u = 0$ on the boundary   $\partial\Pi$.
What distinguishes different fluids is the expression of the stress tensor $\sigma$.
Newtonian fluids are the most encountered ones in real life and as typical examples one can mention the  water and the  air.
For  a newtonian fluid, the entries of the 
 stress tensor $\sigma(p,u)$ are  given by 
\begin{equation}\label{stress-newtonian}
 (\sigma(p,u))_{ij}= -p\delta_{ij} +  2 \mu (D(u))_{ij},\quad 1 \leq i,j \leq 3
 \end{equation}
where $\delta_{ij}$ is the Kronecker symbol, the real positive $\mu$ is the viscosity of the fluid and the entries 
of the strain tensor are 
$(D(u))_{ij}  =   (\partial_{x_j} u_{i} + \partial_{x_i} u_{j})/2.$
If  $f$ belongs to  $(L^2(\Pi))^3$ and the space $V$ is defined by 
$\displaystyle V \,\, = \,\, \{  v \in (H_0^1(\Pi))^3 \,\, | \,\, \mbox{div}\, v = 0 \},$
then $u$ and $p$ satisfying \eqref{motion} with \eqref{stress-newtonian} 
are such that (see for instance \cite{Girault_Raviart}):

{\it (Stokes) There is  a unique $u \in V$ and a unique (up to an additive real constant) 
$ p \in L^2(\Pi)$ such that (if 
$\displaystyle <\cdot, \cdot >$ is the dual pairing between $(H^{-1}(\Pi))^3$ and $(H_0^1(\Pi))^3$)
\begin{equation}\label{Stokes}
a(u,v) =l(v) - <\nabla p, v >, \quad \forall v \in (H_0^1(\Pi))^3,
\end{equation}}
with
$\displaystyle a(u,v) = 2 \mu\int_{\Pi}\hbox {D(u)} \colon \hbox {D(v)} dx$ and 
$\displaystyle l(v)= \int_{\Pi} f \cdot v dx$.

A fluid whose stress is not defined by relation  \eqref{stress-newtonian}  is called  a non-newtonian fluid.
There are several classes of non-newtonian fluids, 
as the power law, Carreau, Cross, Bingham fluids. It is on the study of the last type of fluid that we are 
interested in this paper. We refer to \cite{Cioran-book-fluids} for a review on  non-newtonian fluids.
For  a Bingham fluid, the nonlinear
 stress tensor is  defined by (see \cite{Duvaut_Lions})
 \begin{equation}\label{bingham}
(\sigma(p, u))_{ij}= -p \delta_{ij}  + 
2 \mu (D(u))_{ij}+ {\sqrt{2}} g \frac{(D(u))_{ij}}{\vert D(u) \vert}, 
\end{equation}
where $\displaystyle  \vert D (u) \vert ^2  =   D(u) \colon D(u)$ and the positive number $g$ represents the
yield stress of the fluid. If $g=0$, then \eqref{bingham} becomes \eqref{stress-newtonian}.
Viscoplastic Bingham fluids are quite often encountered in real life.
As examples  one can mention volcanic lava,  fresh concrete, the drilling mud, oils, clays and some paintings.
For $u_g$ and $p_g$ satisfying \eqref{motion} with \eqref{bingham},  
according to \cite{Duvaut_Lions}, one has the following result:

{\it (Bingham) There is  a unique $u_g \in V$ and a (non-unique)  $ p_g \in L^2(\Pi)/\mathbb{R}$ such that
\begin{equation}\label{g}
a(u_g,v-u_g) + j(v)-j(u_g)\, \geq \, l(v-u_g) - <\nabla p_g, v-u_g > , \,\, \forall v \in (H_0^1(\Pi))^3.
\end{equation}}
Here 
$\displaystyle a, l,  <\cdot, \cdot > $  are as before and 
$$j(v)= \sqrt{2} g\int_{\Pi} \vert D(v) \vert dx, \quad  \forall v \in (H_0^1(\Pi))^3.$$

If the yield stress of the Bingham fluid is of the form $g(\varepsilon)$, with $\varepsilon \in ]0,1[$ and such that 
$g (\varepsilon)$
tends to zero when  $\varepsilon$ tends to zero, then,   
 according to [\cite{Duvaut_Lions}, Chapter 6, Th\'eor\`eme 5.1.], the following result holds 

{\it When $\varepsilon$  tends to zero, one has for the solution $u_\varepsilon$ of problem \eqref{g} corresponding to $g(\varepsilon)$ 
the following convergence 
$$u_{\varepsilon} \rightharpoonup u \quad \mbox{ weakly in } V,$$
where $u$ is the solution of problem \eqref{Stokes}.}

This means that, in a fixed domain, the  nonlinear character of the Bingham flow is lost  in the limit (when the yield stress 
tends to zero), as it is expected.
A natural question that arises is the following:   If the yield stress  $g(\varepsilon)$ is as before and, moreover,
the domain $\Pi$ itself depends on the small parameter $\varepsilon$, what happens when  
 $\varepsilon$ tends to zero?
 The answer is that,  in the limit, the nonlinear character of the flow may be preserved. 
 For instance, if $\Pi_\varepsilon$  is a classical rigid porous medium, it was proven in   \cite{Lions_SanchezPalencia}
 with the asymptotic expansion method that, in a range of parameters, the nonlinear character of the Bingham flow is
 preserved in the homogenized problem, which is a nonlinear Darcy equation.
The convergence corresponding to the above mentioned result was  proven in \cite{Bourgeat_Mikelic} with the two-scale 
convergence method  and then recovered in 
  \cite{Bunoiu_Cardone_Perugia} with the periodic unfolding method. 
  The case of a  doubly periodic rigid porous medium was studied in \cite{Bunoiu_Cardone}, where a more involved nonlinear 
  Darcy equation is derived.
  Another class of domains for which the nonlinear character of the flow may be  preserved in the limit is those
  of thin domains. The case of a domain $\Pi_\varepsilon$ which is  thin  in one direction was addressed
  in \cite{Bunoiu_Kesavan-2D} and \cite{Bunoiu_Kesavan}. 
 We refer to \cite{Bunoiu_Gaudiello_Leopardi} for the asymptotic analysis of  a Bingham fluid  in a thin T-like shaped domain.
 In all these cases, a lower-dimensional Bingham-like law was exhibited in the limit. This law was already encountered in 
 engineering (see \cite{Liu-Mei}), but no rigurous mathematical justification was previously known.
 For the  shallow flow of a viscoplastic fluid we refer the reader to \cite{Fernandez}, \cite{Ionescu10}, \cite{Ionescu13bis} and \cite{Ionescu13}.
 
 In this paper we  study the asymptotic behavior of the flow of a viscoplastic Bingham fluid in 
 a thin porous medium.
 We refer the reader to the very recent paper \cite{Saramito} and the references therein for the  
 application of our 
study to problems issued from the real life applications. As  a first example one can mention 
the flow of the volcanic lava  through dense forests (see \cite{Lipman}).  Another important
application 
 is the flow of  fresh concrete spreading through networks of steel bars.
 
 The paper is organized as follows.  In Section 2. we state the problem: we define
in \eqref{Dominio1} the thin porous medium $\Omega_\varepsilon$ (see also Figure 1), of height $\varepsilon$ 
 and  relative dimension $a_\varepsilon$ of its periodically distributed 
pores. In $\Omega_\varepsilon$  we consider the flow
of a viscoplastic Bingham fluid with velocity $u_\varepsilon$ and pressure $p_\varepsilon$ verifying the nonlinear variational
inequality \eqref{VA_pressure}. In Section 3. we give some {\it a priori} estimates for the velocity and 
for the pressure obtained after the change of variables \eqref{dilatacion} and verifying \eqref{VA_pressure2},
and then for the velocity and for the pressure  defined in \eqref{uhat}.
In Section 4. by  passing to the limit $\varepsilon \rightarrow 0$, 
we prove the main convergence results of our paper, stated in 
Theorems \ref{CriticalCase},  \ref{SubCriticalCase} and \ref{SupercriticalCase}, respectively. 
Up to our knowledge, problems \eqref{limit_critical}, \eqref{limit_subcritical} and \eqref{limit_supercritical} 
are new in the mathematical literature. 
We  conclude  in Section 5. with the interpretation of these limit
problems, which all three preserve the  nonlinear character of the flow; both effects of a  nonlinear Darcy equation 
and a  lower dimensional Bingham-like law appear. The paper ends  with a list of References.
 
\section{Statement of the problem}
A periodic porous medium is defined by a domain $\omega$ and an associated microstructure,
or periodic cell $Y^{\prime}=[-1/2,1/2]^2$, which is made of two complementary parts: 
the fluid part $Y^{\prime}_{f}$, and the solid part $Y^{\prime}_{s}$ ($Y^{\prime}_f  \bigcup Y^{\prime}_s=Y^\prime$ 
and $Y^{\prime}_f  \bigcap Y^{\prime}_s=\varnothing$).
More precisely, we assume that $\omega$ is a smooth, bounded, connected set in $\mathbb{R}^2$, and that $Y_s^{\prime}$ 
is an open connected subset of $Y^\prime$ with a smooth boundary $\partial Y_s^\prime$, such that
$\overline Y_s^\prime$ is strictly included  in $Y^\prime$.

The microscale of a porous medium is a small positive number $a_{\varepsilon}$. The domain $\omega$ is covered by a regular mesh of size $a_{\varepsilon}$: for $k^{\prime}\in \mathbb{Z}^2$, each cell $Y^{\prime}_{k^{\prime},a_{\varepsilon}}=a_{\varepsilon}k^{\prime}+a_{\varepsilon}Y^{\prime}$ is divided in a fluid part $Y^{\prime}_{f_{k^{\prime}},a_{\varepsilon}}$ and a solid part $Y^{\prime}_{s_{k^{\prime}},a_{\varepsilon}}$, i.e. is similar to the unit cell $Y^{\prime}$ rescaled to size $a_{\varepsilon}$. We define $Y=Y^{\prime}\times (0,1)\subset \mathbb{R}^3$, which is divided in a fluid part $Y_{f}$ and a solid part $Y_{s}$, and consequently $Y_{k^{\prime},a_{\varepsilon}}=Y^{\prime}_{k^{\prime},a_{\varepsilon}}\times (0,1)\subset \mathbb{R}^3$, which is also divided in a fluid part $Y_{f_{k^{\prime}},a_{\varepsilon}}$ and a solid part $Y_{s_{k^{\prime}},a_{\varepsilon}}$.

We denote by $\tau(\overline Y'_{s_{k'},a_\varepsilon})$ the set of all translated images of
$\overline Y'_{s_{k'},a_\varepsilon}$. The set $\tau(\overline Y'_{s_{k'},a_\varepsilon})$ represents the solids in $\mathbb{R}^2$.
The fluid part of the bottom $\omega_{\varepsilon}\subset \mathbb{R}^2$ of the porous medium is defined by
$\omega_{\varepsilon}=\omega\backslash\bigcup_{k^{\prime}\in \mathcal{K}_{\varepsilon}} \overline Y^{\prime}_{s_{k^{\prime}},
{a_\varepsilon}}$, where 
$\mathcal{K}_{\varepsilon}=\{k^{\prime}\in \mathbb{Z}^2: Y^{\prime}_{k^{\prime}, {a_\varepsilon}} \cap \omega \neq \emptyset \}$.  
The whole fluid part $\Omega_{\varepsilon}\subset \mathbb{R}^3$ in the thin porous medium is defined by 
\begin{equation}\label{Dominio1}
\Omega_{\varepsilon}=\{  (x_1,x_2,x_3)\in \omega_{\varepsilon}\times \mathbb{R}: 0<x_3<\varepsilon \}.
\end{equation}
We make the assumption that the solids  $\displaystyle \tau(\overline Y'_{s_{k'},a_\varepsilon})$ do not intersect the boundary
$\partial \omega$. We define $Y^\varepsilon_{s_{k'},a_\varepsilon}=Y'_{s_{k'},a_\varepsilon}\times (0,\varepsilon)$. Denote by $S_\varepsilon$
the set of the solids contained in $\Omega_\varepsilon$. Then, $S_\varepsilon$ is a finite 
union of solids, i.e. $S_\varepsilon=\bigcup_{k^{\prime}\in \mathcal{K}_{\varepsilon}} \overline Y^\varepsilon_{s_{k'},a_\varepsilon}.$ 

\begin{figure}[ht]
\begin{center}
\includegraphics[width=10cm]{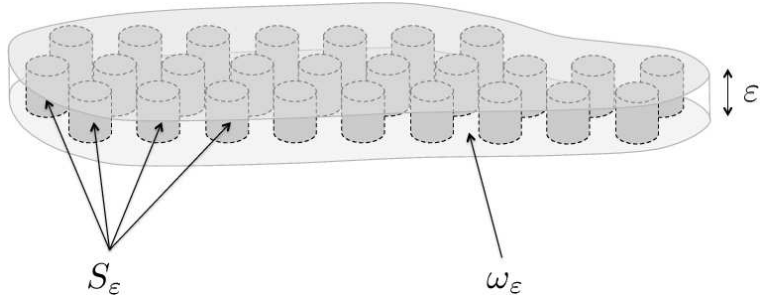}\quad \quad\quad 
\end{center}
\caption{Views of the domain $\Omega_\varepsilon$}
\label{figure_domain}
\end{figure}

We define $\widetilde{\Omega}_{\varepsilon}=\omega_{\varepsilon}\times (0,1),$ $\Omega=\omega\times (0,1),$ and $Q_\varepsilon=\omega\times (0,\varepsilon).$ We observe that $\widetilde{\Omega}_{\varepsilon}=\Omega\backslash \bigcup_{k^{\prime}\in \mathcal{K}_{\varepsilon}}
\overline Y_{s_{k^{\prime}}, {a_\varepsilon}},$ and we define $T_\varepsilon=\bigcup_{k^{\prime}\in \mathcal{K}_{\varepsilon}}
\overline Y_{s_{k^\prime}, a_\varepsilon}$ as the set of the solids contained in $\widetilde \Omega_\varepsilon$.

We denote by $:$ the full contraction of two matrices; for $A=(a_{i,j})_{1\leq i,j\leq 3}$ and $B=(b_{i,j})_{1\leq i,j\leq 3}$, we have $A:B=\sum_{i,j=1}^3 a_{ij}b_{ij}$.

In order to apply the unfolding method, we will need the following notation.  For $k'\in \mathbb{Z}^2$, we define $\kappa:\mathbb{R}^2\to \mathbb{Z}^2$ by 
\begin{equation}\label{kappa_fun}
\kappa(x^\prime)=k^\prime\iff x'\in Y^{\prime}_{k^{\prime},1}\,.
\end{equation}
Remark that $\kappa$ is well defined up to a set of zero measure in $\mathbb{R}^2$ (the set $\cup_{k^\prime\in \mathbb{Z}^2}\partial Y^{\prime}_{k^{\prime},1}$). Moreover, for every $a_\varepsilon>0$, we have 
$$\kappa\left(\frac{x^\prime}{a_\varepsilon}\right)=k^\prime \iff x^\prime\in Y^{\prime}_{k^{\prime},a_\varepsilon}.$$

We denote by $C$ a generic positive constant which can change from line to line.

The points $x\in\mathbb{R}^3$ will be decomposed as $x=(x^{\prime},x_3)$ with $x^{\prime}=(x_1,x_2)\in \mathbb{R}^2$, $x_3\in \mathbb{R}$. We also use the notation $x^{\prime}$ to denote a generic vector of $\mathbb{R}^2$.

In $\Omega_\varepsilon$ we consider the stationary flow of an incompressible Bingham fluid. As already seen in the Introduction, 
following Duvaut and Lions \cite{Duvaut_Lions}, the problem is formulated in terms of a variational inequality.

For a vectorial function $v=(v',v_3)$, we define
$$(D(v))_{i,j}={\frac{1}{2}}\left(\partial_{x_j} v_i+\partial_{x_i} v_j \right),\quad 1\leq i,j\leq 3,\quad |D(v)|^2=D(v):D(v).$$
We introduce the following spaces
$$V(\Omega_\varepsilon)=\{v\in (H_0^1(\Omega_\varepsilon))^3\,|\, {\rm div}\,v=0 \text{ in } \Omega_\varepsilon\}, \quad H(\Omega_\varepsilon)=\{v\in (L^2(\Omega_\varepsilon))^3\,|\, {\rm div}\,v=0 \text{ in } \Omega_\varepsilon, v\cdot n=0 \text{ on } \partial\Omega_\varepsilon\}.$$
For $u,v\in (H_0^1(\Omega_\varepsilon))^3 $, we introduce
$$a(u,v)=2\mu\int_{\Omega_\varepsilon}D(u):D(v)dx,\quad j(v)=\sqrt{2}g(\varepsilon)\int_{\Omega_\varepsilon}|D(v)|dx,\quad (u,v)_{\Omega_\varepsilon}=\int_{\Omega_\varepsilon}u\cdot vdx,$$
where the yield stress $g(\varepsilon)$ will be made precise in Section 3.1.
Let $f\in (L^2(\Omega))^3$ be given such that $f=(f',0)$. Let $f_\varepsilon\in (L^2(\Omega_\varepsilon))^3$ be defined by
$$f_\varepsilon(x)=f(x',x_3/ \varepsilon), \text{\ a.e. \ }x\in \Omega_\varepsilon.$$
The model of the flow is described by the following variational inequality:

Find $u_\varepsilon \in V(\Omega_\varepsilon)$ such that 
\begin{equation} \label{VA}
a(u_\varepsilon,v-u_\varepsilon)+j(v)-j(u_\varepsilon)
\ge 
(f_\varepsilon,v-u_\varepsilon)_ {\Omega_{\varepsilon}},\quad
 \forall v\in V(\Omega_\varepsilon).
\end{equation}

From Duvaut and Lions \cite{Duvaut_Lions}, we know that there exists a unique
$u_\varepsilon\in V(\Omega_\varepsilon)$ solution of problem (\ref{VA}).
Moreover, from Bourgeat and Mikeli\'c \cite{Bourgeat_Mikelic}, we know that if $p_\varepsilon$
is the pressure of the fluid in $\Omega_\varepsilon$, then problem (\ref{VA}) is equivalent to the following one:
Find $u_\varepsilon \in V(\Omega_\varepsilon)$   and $p_\varepsilon \in L_0^2(\Omega_\varepsilon)$  such that 
 \begin{equation}\label{VA_pressure}
a(u_\varepsilon,v-u_\varepsilon)+j(v)-j(u_\varepsilon)
\ge 
(f_\varepsilon,v-u_\varepsilon)_ {\Omega_{\varepsilon}}+(p_\varepsilon,{\rm div}\,(v-u_\varepsilon))_{\Omega_\varepsilon}, \quad
\forall v\in (H_0^1(\Omega_\varepsilon))^3. 
\end{equation}
Problem \eqref{VA_pressure}  admits a unique solution $u_\varepsilon \in V(\Omega_\varepsilon)$ and a (non) unique
solution $p_\varepsilon \in L_0^2(\Omega_\varepsilon)$, where  $L_0^2(\Omega_\varepsilon)$ denotes the space of 
functions belonging to $L^2(\Omega_\varepsilon)$ and of mean value zero. 

Our aim is to study the asymptotic behavior of $u_{\varepsilon}$ and $p_\varepsilon$ when $\varepsilon$ tends to zero. 
For this purpose, we first use the dilatation of the domain $\Omega_\varepsilon$ in the variable $x_3$, namely 
\begin{equation}\label{dilatacion}
y_3=\frac{x_3}{\varepsilon},
\end{equation}
in order to have the functions defined in an open set with fixed height, denoted $\widetilde{\Omega}_{\varepsilon}$.

Namely, we define $\tilde{u}_{\varepsilon}\in (H_0^{1}(\widetilde{\Omega}_{\varepsilon}))^3$, $\tilde p_\varepsilon \in L^2_0(\widetilde\Omega_\varepsilon)$ by $$\tilde{u}_{\varepsilon}(x^{\prime},y_3)=u_{\varepsilon}(x^{\prime},\varepsilon y_3), \quad \tilde{p}_{\varepsilon}(x^{\prime},y_3)=p_{\varepsilon}(x^{\prime},\varepsilon y_3)\text{\ \ } a.e.\text{\ } (x^{\prime},y_3)\in \widetilde{\Omega}_{\varepsilon}.$$
Let us introduce some notation which will be useful in the following. For a vectorial function $v=(v',v_3)$ and a scalar function $w$, we will denote $\mathbb{D}_{x^\prime}\left[v\right]=\frac{1}{2}(D_{x^\prime}v+D_{x^\prime}^t v)$ and $\partial_{y_3}\left[v\right]=\frac{1}{2}(\partial_{y_3}v+\partial_{y_3}^t v)$, where we denote $\partial_{y_3}=(0,0,\frac{\partial}{\partial y_3})^t$. 
Moreover, associated to the change of variables (\ref{dilatacion}), we introduce the operators: 
$D_{\varepsilon}$, $\mathbb{D}_\varepsilon$, ${\rm div}_{\varepsilon}$ and $\nabla_\varepsilon$, defined by
\begin{equation*}
(D_{\varepsilon}v)_{i,j}=\partial_{x_j}v_i\text{\ for \ }i=1,2,3,\ j=1,2,\quad (D_{\varepsilon}v)_{i,3}=\frac{1}{\varepsilon}\partial_{y_3}v_i\text{\ for \ }i=1,2,3,
\end{equation*}
\begin{equation*}
\mathbb{D}_\varepsilon\left[v\right]=\frac{1}{2}\left(D_\varepsilon v+D^t_\varepsilon v \right), \quad |\mathbb{D}_{\varepsilon}\left[v\right]|^2=\mathbb{D}_\varepsilon\left[v\right]:\mathbb{D}_\varepsilon\left[v\right], \quad {\rm div}_{\varepsilon}v={\rm div}_{x^{\prime}}v^{\prime}+\frac{1}{\varepsilon}\partial_{y_3}v_3, \quad \nabla_{\varepsilon}w=(\nabla_{x^{\prime}}w,\frac{1}{\varepsilon}\partial_{y_3}w)^t.
\end{equation*}
We introduce the following spaces
$$V(\widetilde\Omega_\varepsilon)=\{\tilde v\in (H_0^1(\widetilde\Omega_\varepsilon))^3\,|\, {\rm div}_\varepsilon\tilde v=0 \text{ in } \widetilde\Omega_\varepsilon\},\quad H(\widetilde\Omega_\varepsilon)=\{\tilde v\in (L^2(\widetilde \Omega_\varepsilon))^3\,|\, {\rm div}_\varepsilon\tilde v=0 \text{ in } \widetilde\Omega_\varepsilon, \tilde v\cdot n=0 \text{ on } \partial\widetilde \Omega_\varepsilon\}.$$
For $\tilde u,\tilde v\in V(\widetilde\Omega_\varepsilon)$, we introduce
$$a_\varepsilon(\tilde u,\tilde v)=2\mu\int_{\widetilde\Omega_\varepsilon}\mathbb{D}_\varepsilon\left[\tilde u\right]:\mathbb{D}_\varepsilon\left[\tilde v\right]dx'dy_3,\quad j_{\varepsilon}(\tilde v)=\sqrt{2}g(\varepsilon)\int_{\widetilde\Omega_\varepsilon}|\mathbb{D}_{\varepsilon}[\tilde v]|dx'dy_3,\quad (\tilde u,\tilde v)_{\widetilde\Omega_\varepsilon}=\int_{\widetilde\Omega_\varepsilon}\tilde u\cdot \tilde vdx'dy_3.$$
Using the transformation (\ref{dilatacion}), the variational inequality (\ref{VA}) can be rewritten as: 

Find $ \tilde u_\varepsilon \in V(\widetilde\Omega_\varepsilon)$ such that 
\begin{equation} \label{VA2}
a_\varepsilon(\tilde u_\varepsilon,\tilde v-\tilde u_\varepsilon)+j_\varepsilon(\tilde v)-j_\varepsilon(\tilde u_\varepsilon)
\ge 
(f,\tilde v-\tilde u_\varepsilon)_ {\widetilde\Omega_{\varepsilon}},\quad
 \forall \tilde v\in V(\widetilde\Omega_\varepsilon), 
\end{equation}
and (\ref{VA_pressure}) can be rewritten as:

Find $\tilde u_\varepsilon \in V(\widetilde \Omega_\varepsilon)$   and   $\tilde p_\varepsilon \in L_0^2(\widetilde \Omega_\varepsilon)$ such that 
 \begin{equation}\label{VA_pressure2}
a_\varepsilon(\tilde u_\varepsilon,\tilde v-\tilde u_\varepsilon)+j_\varepsilon(\tilde v)-j_\varepsilon(\tilde u_\varepsilon)
\ge 
(f,\tilde v-\tilde u_\varepsilon)_ {\widetilde\Omega_{\varepsilon}}+(\tilde p_\varepsilon,{\rm div}_{\varepsilon}(\tilde v-\tilde u_\varepsilon))_{\widetilde\Omega_\varepsilon}, \quad \forall \tilde v\in (H_0^1(\widetilde \Omega_\varepsilon))^3.
\end{equation}
Our goal now  is to describe the asymptotic behavior of this new sequence $(\tilde{u}_{\varepsilon}$, $\tilde{p}_{\varepsilon})$. 

\section{A Priori Estimates}
We start by  obtaining  some {\it a priori} estimates for $\tilde{u}_{\varepsilon}$.
\begin{lemma}\label{Lemma_estimate}
There exists a constant $C$ independent of $\varepsilon$, such that if 
$\tilde{u}_{\varepsilon}\in (H_0^{1}(\widetilde{\Omega}_{\varepsilon}))^3$ is the solution of problem (\ref{VA2}), one has
\begin{itemize}
\item[i)] if $a_{\varepsilon}\approx \varepsilon$, with $a_\varepsilon/\varepsilon\to \lambda$, $0<\lambda<+\infty$,
or $a_\varepsilon\ll \varepsilon$, then 
\begin{equation}\label{a1}
\left\Vert \tilde{u}_{\varepsilon}\right\Vert_{(L^2(\widetilde{\Omega}_{\varepsilon}))^3}\leq 
\frac{C}{\mu}a_{\varepsilon}^2,\quad \left\Vert \mathbb{D}_{\varepsilon}\left[\tilde{u}_{\varepsilon}\right]\right\Vert_{(L^2(\widetilde{\Omega}_{\varepsilon}))^{{3\times3}}}\leq 
\frac{C}{\mu}a_{\varepsilon}, \quad \left\Vert D_{\varepsilon}\tilde{u}_{\varepsilon}\right\Vert_{(L^2(\widetilde{\Omega}_{\varepsilon}))^{{3\times3}}}\leq 
\frac{C}{\mu}a_{\varepsilon},
\end{equation}
\item[ii)] if $a_\varepsilon\gg \varepsilon$, then 
\begin{equation}\label{a2}
\left\Vert \tilde{u}_{\varepsilon}\right\Vert_{(L^2(\widetilde{\Omega}_{\varepsilon}))^3}\leq 
\frac{C}{\mu}\varepsilon^2,\quad \left\Vert \mathbb{D}_{\varepsilon}\left[\tilde{u}_{\varepsilon}\right]\right\Vert_{(L^2(\widetilde{\Omega}_{\varepsilon}))^{{3\times3}}}\leq \frac{C}{\mu}\varepsilon, \quad \left\Vert D_{\varepsilon}\tilde{u}_{\varepsilon}\right\Vert_{(L^2(\widetilde{\Omega}_{\varepsilon}))^{{3\times3}}}\leq \frac{C}{\mu}\varepsilon.
\end{equation}
\end{itemize}
\end{lemma}
\begin{proof}
Setting successively $\tilde v=2\tilde u_\varepsilon$ and $\tilde v=0$  in (\ref{VA2}), we have
\begin{eqnarray}\label{d1}
2\mu\int_{\widetilde{\Omega}_{\varepsilon}}\mathbb{D}_{\varepsilon}\left[\tilde{u}_{\varepsilon}\right]:\mathbb{D}_{\varepsilon}\left[\tilde{u}_{\varepsilon}\right]dx^{\prime}dy_3+\sqrt{2}g(\varepsilon)\int_{\widetilde{\Omega}_{\varepsilon}}|\mathbb{D}_{\varepsilon}[\tilde u_\varepsilon]|dx'dy_3=\int_{\widetilde{\Omega}_{\varepsilon}}f\cdot\tilde{u}_{\varepsilon}\,dx^{\prime}dy_3.
\end{eqnarray}
Using Cauchy-Schwarz's inequality and the assumption of $f$, we obtain that 
\begin{eqnarray*}
\int_{\widetilde{\Omega}_{\varepsilon}}f\cdot\tilde{u}_{\varepsilon}\,dx^{\prime}dy_3\leq C \left\Vert\tilde{u}_{\varepsilon} \right\Vert_{(L^2(\widetilde{\Omega}_{\varepsilon}))^3},
\end{eqnarray*}
and taking into account that $\int_{\widetilde{\Omega}_{\varepsilon}}|\mathbb{D}_{\varepsilon}[\tilde u_\varepsilon]|dx'dy_3\ge 0,$
by (\ref{d1}), we have
\begin{equation*}\label{d2}
\left\Vert \mathbb{D}_{\varepsilon}\left[\tilde{u}_{\varepsilon}\right]\right\Vert_{(L^2(\widetilde{\Omega}_{\varepsilon}))^{{3\times3}}}^2\leq 
\frac{C}{\mu} \left\Vert\tilde{u}_{\varepsilon} \right\Vert_{(L^2(\widetilde{\Omega}_{\varepsilon}))^3}.
\end{equation*}
For the cases $a_\varepsilon \approx \varepsilon$ or $a_\varepsilon \ll \varepsilon$, taking into account Remark 4.3(i) in \cite{Anguiano_SuarezGrau}, we obtain the second estimate in (\ref{a1}), and, consequently, from classical Korn's inequality we obtain the last estimate in (\ref{a1}). Now, from the second estimate in (\ref{a1}) and Remark 4.3(i) in \cite{Anguiano_SuarezGrau}, we deduce the first estimate in (\ref{a1}). For the case $a_\varepsilon \gg \varepsilon$, proceeding similarly with Remark 4.3(ii) in \cite{Anguiano_SuarezGrau}, we obtain the desired result.
\end{proof}

\subsection{The extension of $(\tilde u_\varepsilon,\tilde p_\varepsilon)$ to the whole domain $\Omega$}
We extend the velocity $\tilde u_\varepsilon$ by zero to the $\Omega\backslash \widetilde \Omega_\varepsilon$ and 
denote the extension by the same symbol. Obviously,  estimates (\ref{a1})-(\ref{a2}) remain valid and the extension is divergence free too.

We study in the sequel the following  cases for the value of  yield stress  $g (\varepsilon)$:
\begin{itemize}
\item[i)] if $a_{\varepsilon}\approx \varepsilon$, with $a_\varepsilon/\varepsilon\to \lambda$, $0<\lambda<+\infty$,
or $a_\varepsilon\ll \varepsilon$, then $\displaystyle g(\varepsilon)=g\,a_\varepsilon,$
\item[ii)] if $a_\varepsilon\gg \varepsilon$, then $\displaystyle g (\varepsilon)=g\,\varepsilon.$
\end{itemize}
These choices are the most challenging ones and they answer to the question adressed in the paper,
namely they all preserve in the limit the nonlinear character of the flow.

In order to extend the pressure to the whole domain $\Omega$, the mapping $R^\varepsilon$ 
(defined in Lemma 4.5 in \cite{Anguiano_SuarezGrau} as $R_2^\varepsilon$) allows us to extend 
the pressure $p_\varepsilon$ to $Q_\varepsilon$ by introducing $F_\varepsilon$ in $(H^{-1}(Q_\varepsilon))^3$:
\begin{equation}\label{F}
\left\langle F_{\varepsilon}, w\right\rangle_{Q_\varepsilon}=\left\langle \nabla{p}_{\varepsilon}, R^{\varepsilon} w\right\rangle_{{\Omega}_{\varepsilon}}, \text{\ for any \ } w\in (H_0^{1}(Q_\varepsilon))^3.
\end{equation}
Setting succesively $ v=  u_\varepsilon + R^{\varepsilon} w$ and $v= u_\varepsilon - R^{\varepsilon} w $ in (\ref{VA_pressure}) we get the inequality
\begin{equation}\label{inequality}
|\left\langle F_{\varepsilon}, w\right\rangle_{Q_\varepsilon}|\leq| a( u_\varepsilon, R^{\varepsilon} w)|+|( f_\varepsilon, R^{\varepsilon} w)_{{\Omega}_{\varepsilon}}|+j( R^{\varepsilon}w).
\end{equation}
Moreover, if ${\rm div}\,w=0$ then
$\left\langle F_{\varepsilon}, w\right\rangle_{Q_\varepsilon}=0,$
and the DeRham Theorem gives the existence of $P_\varepsilon$ in $L_0^2(Q_\varepsilon)$ with $F_\varepsilon=\nabla P_\varepsilon$.

Using the change of variables (\ref{dilatacion}), we get for any $\tilde w\in (H_0^1(\Omega))^3$ where $\tilde w(x',y_3)=w(x',\varepsilon y_3)$, 
$$\left\langle \nabla_\varepsilon \tilde P_\varepsilon, \tilde w\right\rangle_{\Omega}=-\int_{\Omega}\tilde P_\varepsilon\, {\rm div}_{\varepsilon}\tilde w\,dx'dy_3=-\varepsilon^{-1}\int_{Q_\varepsilon} P_\varepsilon\, {\rm div}\, w\,dx=\varepsilon^{-1}\left\langle \nabla  P_\varepsilon,  w\right\rangle_{Q_\varepsilon}.$$
Then, using the identification (\ref{F}) of $F_\varepsilon$ and the inequality (\ref{inequality}),
\begin{equation*}\label{inequality2}
|\left\langle \nabla_\varepsilon \tilde P_\varepsilon, \tilde w\right\rangle_{\Omega}|\leq\varepsilon^{-1}\left(| a( u_\varepsilon, R^{\varepsilon} w)|+|( f_\varepsilon, R^{\varepsilon} w)_{{\Omega}_{\varepsilon}}|+j( R^{\varepsilon}w)\right).
\end{equation*}
and applying the change of variables (\ref{dilatacion}),
\begin{equation}\label{inequality_duality}
|\left\langle \nabla_\varepsilon \tilde P_\varepsilon,\tilde w\right\rangle_{\Omega}|\leq| a_\varepsilon(\tilde u_\varepsilon,\tilde R^{\varepsilon}\tilde w)|+|( f,\tilde R^{\varepsilon}\tilde w)_{\widetilde{\Omega}_{\varepsilon}}|+j_\varepsilon(\tilde R^{\varepsilon}\tilde w),
\end{equation}
where $\tilde R^{\varepsilon}\tilde w=R^\varepsilon w$ for any $\tilde w\in (H_0^1(\Omega))^3$.

Now, we estimate the right-hand side of (\ref{inequality_duality}) using the estimates given in Lemma 4.6 in \cite{Anguiano_SuarezGrau}.

\begin{lemma}
There exists a constant $C$ independent of $\varepsilon$, such that the extension $\tilde{P}_{\varepsilon}\in L_0^{2}(\Omega)$ of the pressure $ \tilde{p}_{\varepsilon}$ satisfies
\begin{equation}\label{esti_P}
\left\Vert \tilde{P}_{\varepsilon} \right\Vert_{L_0^{2}(\Omega)}\leq C.
\end{equation}
\end{lemma}
\begin{proof}
Let us estimate $\nabla_{\varepsilon} \tilde{P}_{\varepsilon}$ in the cases $a_{\varepsilon}\approx \varepsilon$ or $a_\varepsilon\ll \varepsilon$. We estimate the right-hand side of (\ref{inequality_duality}). Using Cauchy-Schwarz's inequality and from the second estimate in (\ref{a1}) we have
\begin{eqnarray*}
| a_\varepsilon(\tilde u_\varepsilon,\tilde R^{\varepsilon}\tilde w)|&\leq& 2\mu\left\Vert \mathbb{D}_{\varepsilon}\left[\tilde{u}_{\varepsilon}\right]\right\Vert_{(L^2(\widetilde{\Omega}_{\varepsilon}))^{3\times3}}\left\Vert D_{\varepsilon}\tilde R^{\varepsilon}\tilde w\right\Vert_{(L^2(\widetilde{\Omega}_{\varepsilon}))^{3\times3}}\leq Ca_{\varepsilon}\left\Vert D_{\varepsilon} \tilde R^{\varepsilon}\tilde w\right\Vert_{(L^2(\widetilde{\Omega}_{\varepsilon}))^{3\times3}}.
\end{eqnarray*}
Using the assumption made on the function $f$, we obtain
\begin{eqnarray*}
|( f,\tilde R^{\varepsilon}\tilde w)_{\widetilde{\Omega}_{\varepsilon}}|\leq C\left\Vert \tilde R^{\varepsilon}\tilde w\right\Vert_{(L^2(\widetilde{\Omega}_{\varepsilon}))^3},
\end{eqnarray*}
and by Cauchy-Schwarz's inequality and taking into account that $|\widetilde{\Omega}_{\varepsilon}|\leq |\Omega|$, we obtain
\begin{eqnarray*}
j_\varepsilon(\tilde R^{\varepsilon}\tilde w)\leq C\,a_\varepsilon \left\Vert D_{\varepsilon} \tilde R^{\varepsilon}\tilde w\right\Vert_{(L^2(\widetilde{\Omega}_{\varepsilon}))^{3\times3}}.
\end{eqnarray*}
Then, from (\ref{inequality_duality}), we deduce
\begin{equation*}
\left\vert\left\langle \nabla_{\varepsilon}\tilde{P}_{\varepsilon},\tilde w\right\rangle_{\Omega}\right\vert\leq
C  a_{\varepsilon} \left\Vert D_{\varepsilon} \tilde R^{\varepsilon}\tilde w \right\Vert_{(L^2(\widetilde{\Omega}_{\varepsilon}))^{3\times3}}+C\left\Vert \tilde R^{\varepsilon}\tilde w\right\Vert_{(L^2(\widetilde{\Omega}_{\varepsilon}))^3}.
\end{equation*}
Taking into account the third point in Lemma 4.6 in \cite{Anguiano_SuarezGrau}, we have 
\begin{eqnarray*}
\left\vert \left\langle \nabla_{\varepsilon}\tilde{P}_{\varepsilon},\tilde w\right\rangle_{\Omega}\right\vert&\leq&
C a_{\varepsilon} \left(\frac{1}{a_{\varepsilon}}\left\Vert \tilde w\right\Vert_{(L^2(\Omega))^3}+\left\Vert D_{\varepsilon} \tilde w\right\Vert_{(L^2(\Omega))^{3\times3}} \right)+ C \left(\left\Vert \tilde w\right\Vert_{(L^2(\Omega))^3}+a_{\varepsilon}\left\Vert D_{\varepsilon} \tilde w\right\Vert_{(L^2(\Omega))^{3\times3}}\right).\nonumber
\end{eqnarray*}
If $a_{\varepsilon}\approx \varepsilon$ we take into account that $a_{\varepsilon}\ll1$, and if $a_\varepsilon\ll \varepsilon$ we take into account that $a_\varepsilon/\varepsilon\ll 1$ and $a_{\varepsilon}\ll1$, and we see that there exists a positive constant $C$ such that  
\begin{eqnarray*}
\left\vert \left\langle \nabla_{\varepsilon}\tilde{P}_{\varepsilon},\tilde w\right\rangle_{\Omega}\right\vert&\leq& C\left\Vert \tilde w\right\Vert_{(H_0^{1}(\Omega))^3}, \quad \forall \tilde w\in (H_0^{1}(\Omega))^3,
\end{eqnarray*}
and consequently
$$\left\Vert\nabla_{\varepsilon}\tilde{P}_{\varepsilon}\right\Vert_{(H^{-1}(\Omega))^3}\leq C.$$
It follows that (see for instance Girault and Raviart \cite{Girault_Raviart}, Chapter I, Corollary 2.1) there exists a representative of $\tilde P_\varepsilon\in L^2_0(\Omega)$ such that
$$\left\Vert \tilde P_\varepsilon\right\Vert_{L^2_0(\Omega)}\leq C \left\Vert\nabla \tilde{P}_{\varepsilon}\right\Vert_{(H^{-1}(\Omega))^3}\leq C \left\Vert\nabla_{\varepsilon}\tilde{P}_{\varepsilon}\right\Vert_{(H^{-1}(\Omega))^3}\leq C.$$

Finally, let us estimate $\nabla_{\varepsilon} \tilde{P}_{\varepsilon}$ in the case $a_\varepsilon\gg \varepsilon$.
Similarly to the previous case, we estimate the right side of (\ref{inequality_duality}) by using Cauchy-Schwarz's 
inequality and from the second estimate in (\ref{a2}), and we have
\begin{equation*}
\left\vert\left\langle \nabla_{\varepsilon}\tilde{P}_{\varepsilon},\tilde w\right\rangle_{\Omega}\right\vert\leq 
C  \varepsilon \left\Vert D_{\varepsilon} \tilde R^{\varepsilon}\tilde w \right\Vert_{(L^2(\widetilde{\Omega}_{\varepsilon}))^{3\times3}}+C\left\Vert \tilde R^{\varepsilon}\tilde w\right\Vert_{(L^2(\widetilde{\Omega}_{\varepsilon}))^3}.
\end{equation*}
Taking into account the proof in Lemma 4.5 in \cite{Anguiano_SuarezGrau}, the change of variables (\ref{dilatacion}) and that $a_\varepsilon\gg \varepsilon$, we can deduce 
\begin{eqnarray*}
\left\vert \left\langle \nabla_{\varepsilon}\tilde{P}_{\varepsilon},\tilde w\right\rangle_{\Omega}\right\vert&\leq&
C  \varepsilon \left({1\over \varepsilon}\left\Vert \tilde w\right\Vert_{(L^2(\Omega))^3}+{1\over \varepsilon}\left\Vert D_{x'} \tilde w\right\Vert_{(L^2(\Omega))^{3\times2}}+{1\over \varepsilon}\left\Vert \partial_{y_3} \tilde w\right\Vert_{(L^2(\Omega))^{3}} \right)\\
&+& C \left(\left\Vert \tilde w\right\Vert_{(L^2(\Omega))^3}+a_{\varepsilon}\left\Vert D_{x'} \tilde w\right\Vert_{(L^2(\Omega))^{3\times2}}+\left\Vert \partial_{y_3} \tilde w\right\Vert_{(L^2(\Omega))^{3}}\right),
\end{eqnarray*}
and using that $a_{\varepsilon}\ll1$, we see that there exists a positive constant $C$ such that  
\begin{eqnarray*}
\left\vert \left\langle \nabla_{\varepsilon}\tilde{P}_{\varepsilon},\tilde w\right\rangle_{\Omega}\right\vert&\leq& C\left\Vert \tilde w\right\Vert_{(H_0^{1}(\Omega))^3}, \quad \forall \tilde w\in (H_0^{1}(\Omega))^3,
\end{eqnarray*}
and reasing as the previous case, we have the estimate (\ref{esti_P}).
\end{proof}

According to these extensions, problem (\ref{VA_pressure2}) can be written as:
\begin{eqnarray}\label{extension_problem}
&&2\mu \int_{\Omega}\mathbb{D}_{\varepsilon}\left[\tilde{u}_{\varepsilon}\right]:\mathbb{D}_{\varepsilon}\left[\tilde v-\tilde{u}_{\varepsilon}\right]dx^{\prime}dy_3+\sqrt{2}g (\varepsilon) \int_{\Omega}|\mathbb{D}_{\varepsilon}[\tilde v]|dx'dy_3-\sqrt{2}g(\varepsilon)\int_{\Omega}|\mathbb{D}_{\varepsilon}[\tilde u_\varepsilon]|dx'dy_3\\
&&\ge\int_{\Omega}f\cdot(\tilde v-\tilde{u}_{\varepsilon})\,dx^{\prime}dy_3+\int_{\Omega}\tilde P_\varepsilon\, {\rm div}_\varepsilon (\tilde v-\tilde u_\varepsilon)dx'dy_3, \nonumber
\end{eqnarray}
for every $\tilde v$ that is the extension by zero to the whole $\Omega$ of a function in $(H_0^1(\widetilde \Omega_\varepsilon))^3$.

\section{Adaptation of the Unfolding Method}
The change of variable (\ref{dilatacion}) does not provide the information we need about the behavior 
of $\tilde{u}_{\varepsilon}$ in the microstructure associated to $\widetilde{\Omega}_{\varepsilon}$.
To solve this difficulty, we use an adaptation introduced in \cite{Anguiano_SuarezGrau} of the unfolding 
method from \cite{CDG}.

Let us recall this adaptation of the unfolding method in which we divide the domain $\Omega$ in cubes of lateral length $a_\varepsilon$ and vertical length $1$. For this purpose, given $(\tilde{u}_{\varepsilon}, \tilde{P}_{\varepsilon})\in (H_0^{1}(\Omega))^3\times L^{2}_0(\Omega)$, we define $(\hat{u}_{\varepsilon}, \hat{P}_{\varepsilon})$ by
\begin{eqnarray}\label{uhat}
\hat{u}_{\varepsilon}(x^{\prime},y)=\tilde{u}_{\varepsilon}\left( a_{\varepsilon}\kappa\left(\frac{x^{\prime}}{a_{\varepsilon}} \right)+a_{\varepsilon}y^{\prime},y_3 \right),\quad \hat{P}_{\varepsilon}(x^{\prime},y)=\tilde{P}_{\varepsilon}\left( a_{\varepsilon}\kappa\left(\frac{x^{\prime}}{a_{\varepsilon}} \right)+a_{\varepsilon}y^{\prime},y_3 \right),\text{\ \ a.e. \ }(x^{\prime},y)\in \omega\times Y,
\end{eqnarray}
where the function $\kappa$ is defined in (\ref{kappa_fun}).

\begin{remark}\label{remarkCV}
For $k^{\prime}\in \mathcal{K}_{\varepsilon}$, the restriction of $(\hat{u}_{\varepsilon}, \hat{P}_{\varepsilon})$ to $Y^{\prime}_{k^{\prime},a_{\varepsilon}}\times Y$ does not depend on $x^{\prime}$, whereas as a function of $y$ it is obtained from $(\tilde{u}_{\varepsilon}, \tilde{P}_{\varepsilon})$ by using the change of variables $\displaystyle y^{\prime}=\frac{x^{\prime}-a_{\varepsilon}k^{\prime}}{a_{\varepsilon}},$
which transforms $Y_{k^{\prime},a_{\varepsilon}}$ into $Y$.
\end{remark}
We are now in position to  obtain  estimates for the sequences $(\hat{u}_{\varepsilon}, \hat{P}_{\varepsilon})$, as in the proof
of Lemma 4.9 in \cite{Anguiano_SuarezGrau}.
\begin{lemma}\label{estCV}
There exists a constant $C$ independent of $\varepsilon$, such that the couple $(\hat{u}_{\varepsilon}, \hat{P}_{\varepsilon})$ defined by (\ref{uhat}) satisfies
\begin{itemize}
\item[i)] if $a_{\varepsilon}\approx \varepsilon$, with $a_\varepsilon/\varepsilon\to \lambda$, $0<\lambda<+\infty$, or $a_\varepsilon\ll \varepsilon$,
\begin{equation*}\label{UM1}
\left\Vert \hat{u}_{\varepsilon}\right\Vert_{(L^2(\omega\times Y))^3}\leq Ca_{\varepsilon}^{2},\quad \left\Vert \mathbb{D}_{y^{\prime}}\!\left[\hat{u}_{\varepsilon}\right]\right\Vert_{(L^2(\omega\times Y))^{{3\times2}}}\!\leq\! Ca_{\varepsilon}^{2},\quad \left\Vert \partial_{y_3}\!\left[\hat{u}_{\varepsilon}\right]\right\Vert_{(L^2(\omega\times Y))^3}\!\leq\! C\varepsilon\, a_{\varepsilon},
\end{equation*}
\item[ii)] if $a_\varepsilon\gg \varepsilon$, 
\begin{equation*}\label{UM3}
\left\Vert \hat{u}_{\varepsilon}\right\Vert_{(L^2(\omega\times Y))^3}\leq C\varepsilon^{2},\quad \left\Vert \mathbb{D}_{y^{\prime}}\!\left[\hat{u}_{\varepsilon}\right]\right\Vert_{(L^2(\omega\times Y))^{{3\times2}}}\!\leq\! Ca_{\varepsilon}\,\varepsilon,\quad \left\Vert \partial_{y_3}\!\left[\hat{u}_{\varepsilon}\right]\right\Vert_{(L^2(\omega\times Y))^3}\!\leq\! C\varepsilon^{2},
\end{equation*}
\end{itemize}
and, moreover, in every cases,
\begin{equation*}\label{UM5}
\left\Vert \hat{P}_{\varepsilon} \right\Vert_{L^{2}_0(\omega\times Y)}\leq C.
\end{equation*}
\end{lemma}

When $\varepsilon$ tends to zero, we obtain for problem \eqref{extension_problem} different behaviors,
depending on the magnitude of $a_\varepsilon$ with respect to $\varepsilon$. 
We will analyze them in the next sections.

\subsection{Critical case $a_{\varepsilon}\approx \varepsilon$, with $a_\varepsilon/\varepsilon\to \lambda$, $0<\lambda<+\infty$}
First, we obtain some compactness results about the behavior of the sequences $(\tilde u_\varepsilon, \tilde P_\varepsilon)$ and $(\hat u_\varepsilon, \hat P_\varepsilon)$ satisfying the {\it a priori} estimates given in Lemmas \ref{Lemma_estimate}-i) and \ref{estCV}-i), respectively. 
\begin{lemma}[Critical case]\label{Convergence_critical}
For a subsequence of $\varepsilon$ still denote by $\varepsilon$, there exist $\tilde{u}\in H^1(0,1;L^2(\omega)^3)$, where $\tilde u_3=0$ and $\tilde{u}=0$ on $y_3=\{0,1\}$, $\hat{u}\in L^2(\omega;H^1_{\sharp}(Y)^3)$ (``$\sharp$'' denotes $Y'$-periodicity), with $\hat{u}=0$ on $\omega\times Y_s$ and $\hat u=0$ on $y_3=\{0,1\}$ such that $\int_{Y}\hat u(x',y)dy=\int_0^1 \tilde u(x',y_3)dy_3$ with $\int_Y \hat u_3dy=0$, and $\hat{P}\in L^{2}_0(\omega \times Y)$, independent of $y$, such that
\begin{equation}\label{convtilde1}
{\tilde{u}_{\varepsilon} \over a_\varepsilon^2}\rightharpoonup (\tilde{u}^\prime,0)\text{\ in \ }H^1(0,1;L^2(\omega)^3),
\end{equation}
\begin{equation}\label{convUgorro1}
{\hat{u}_{\varepsilon}\over a_\varepsilon^2}\rightharpoonup \hat{u}\text{\ in \ }L^2(\omega;H^1(Y)^3), \quad \hat{P}_{\varepsilon}\rightharpoonup \hat{P}\text{\ in \ }L^{2}_0(\omega \times Y),
\end{equation}
\begin{equation}\label{divmacro_tilde1}
{\rm div}_{x^{\prime}}\left( \int_0^1 \tilde{u}^{\prime}(x^{\prime},y_3)dy_3 \right)=0 \text{\ in \ }\omega,\quad \left( \int_0^1 \tilde{u}^{\prime}(x^{\prime},y_3)dy_3 \right)\cdot n=0\text{\ on \ }\partial \omega,
\end{equation}
\begin{equation}\label{divmacro1}
{\rm div}_{\lambda}\hat{u}=0 \text{\ in \ }\omega\times Y,\quad {\rm div}_{x^{\prime}}\left( \int_Y \hat{u}^{\prime}(x^{\prime},y)dy \right)=0 \text{\ in \ }\omega,\quad \left( \int_Y \hat{u}^{\prime}(x^{\prime},y)dy \right)\cdot n=0\text{\ on \ }\partial \omega,
\end{equation}
where ${\rm div}_{\lambda}={\rm div}_{y^\prime}+\lambda \partial_{y_3}$.
\end{lemma}
\begin{proof}
We refer the reader to  Lemmas 5.2, 5.3 and 5.4 in \cite{Anguiano_SuarezGrau} for the proof of
(\ref{convtilde1})-(\ref{divmacro1}). Here, we prove that $\hat P$ does not depend on the microscopic variable $y$.
To do this, we choose as test function $\tilde v(x^{\prime},y)\in \mathcal{D}(\omega;C_{\sharp}^{\infty}(Y)^3)$ with $\tilde v(x^{\prime},y)=0\in \omega \times Y_s$ (thus, $\tilde
v(x^{\prime},x^{\prime}/a_{\varepsilon},y_3)\in (H_0^{1}(\widetilde{\Omega}_{\varepsilon}))^3$). 
Setting $a_\varepsilon\tilde v(x',x'/a_\varepsilon, y_3)$ in (\ref{extension_problem}) (we recall that $g(\varepsilon)= g\, a_\varepsilon)$) and using 
that ${\rm div}_\varepsilon \tilde u_\varepsilon=0$, we have
\begin{eqnarray}\label{problem_pressure1}
&&2\mu a_\varepsilon\int_{\Omega}\mathbb{D}_\varepsilon \left[\tilde{u}_\varepsilon \right] :\left(\mathbb{D}_{x^\prime}\left[\tilde v\right]+\frac{1}{a_\varepsilon} \mathbb{D}_{y^\prime}\left[\tilde v\right]+\frac{1}{\varepsilon}\partial_{y_3}\left[\tilde v\right]\right) dx^{\prime}dy_3-2\mu\int_{\Omega}|\mathbb{D}_\varepsilon \left[\tilde{u}_\varepsilon \right]|^2dx'dy_3 \\
&&+\sqrt{2}g\,a_\varepsilon^2 \int_{\Omega}\left|\mathbb{D}_{x^\prime}\left[\tilde v\right]+\frac{1}{a_\varepsilon} \mathbb{D}_{y^\prime}\left[\tilde v\right]+\frac{1}{\varepsilon}\partial_{y_3}\left[\tilde v\right]\right|dx'dy_3-\sqrt{2}g\,a_\varepsilon \int_{\Omega}|\mathbb{D}_{\varepsilon}[\tilde u_\varepsilon]|dx'dy_3 \nonumber\\
&&\ge a_\varepsilon \int_{\Omega}f'\cdot\tilde v'\,dx^{\prime}dy_3-\int_{\Omega}f'\cdot\tilde{u}'_{\varepsilon}\,dx^{\prime}dy_3+a_\varepsilon\int_{\Omega} \tilde{P}_\varepsilon\, {\rm div}_{x^\prime} \tilde v^\prime\,dx^{\prime}dy_3+\int_{\Omega} \tilde{P}_\varepsilon\, {\rm div}_{y^\prime} \tilde v^\prime\,dx^{\prime}dy_3 +\frac{a_\varepsilon}{\varepsilon}\int_{\Omega}\tilde{P}_\varepsilon\,\partial_{y_3}\tilde v_3\,dx^{\prime}dy_3. \nonumber
\end{eqnarray}
By the change of variables given in Remark \ref{remarkCV} and by Lemma \ref{estCV}, we get for the first term in relation (\ref{problem_pressure1})
\begin{eqnarray}\label{term1_CV}
&&\int_{\Omega}\mathbb{D}_\varepsilon \left[\tilde{u}_\varepsilon \right] :\left(\mathbb{D}_{x^\prime}\left[\tilde v\right]+\frac{1}{a_\varepsilon} \mathbb{D}_{y^\prime}\left[\tilde v\right]+\frac{1}{\varepsilon}\partial_{y_3}\left[\tilde v\right]\right) dx^{\prime}dy_3\\
&&=\int_{\omega\times Y}\left(\frac{1}{a_\varepsilon}\mathbb{D}_{y^\prime} \left[\hat{u}_\varepsilon \right] +\frac{1}{\varepsilon}\partial_{y_3}\left[\hat{u}_\varepsilon \right]\right):\left(\frac{1}{a_\varepsilon} \mathbb{D}_{y^\prime}\left[\tilde v\right]+\frac{1}{\varepsilon}\partial_{y_3}\left[\tilde v\right]\right)dx^{\prime}dy+ O_\varepsilon,\nonumber
\end{eqnarray}
and for the second term in relation (\ref{problem_pressure1})
\begin{eqnarray}\label{term0_CV}
\int_{\Omega}|\mathbb{D}_\varepsilon \left[\tilde{u}_\varepsilon \right]|^2dx'dy_3=\int_{\omega\times Y}\left|\frac{1}{a_\varepsilon} \mathbb{D}_{y^\prime}\left[\hat u_\varepsilon\right]+\frac{1}{\varepsilon}\partial_{y_3}\left[\hat u_\varepsilon\right]\right|^2dx'dy=O_\varepsilon.
\end{eqnarray}
Moreover, applying the change of variables given in Remark \ref{remarkCV} to the fourth term in relation (\ref{problem_pressure1}), we have
\begin{eqnarray}\label{term2_CV}
&&\int_{\Omega}|\mathbb{D}_{\varepsilon}[\tilde u_\varepsilon]|dx'dy_3=\int_{\omega\times Y}\left|\frac{1}{a_\varepsilon} \mathbb{D}_{y^\prime}\left[\hat u_\varepsilon\right]+\frac{1}{\varepsilon}\partial_{y_3}\left[\hat u_\varepsilon\right]\right|dx'dy.
\end{eqnarray}
Therefore, applying the change of variables given in Remark \ref{remarkCV} to relation (\ref{problem_pressure1}), we obtain
\begin{eqnarray}\label{problem2_pressure1}
&&2\mu a_\varepsilon\int_{\omega\times Y}\left(\frac{1}{a_\varepsilon}\mathbb{D}_{y^\prime} \left[\hat{u}_\varepsilon \right] +\frac{1}{\varepsilon}\partial_{y_3}\left[\hat{u}_\varepsilon \right]\right):\left(\frac{1}{a_\varepsilon} \mathbb{D}_{y^\prime}\left[\tilde v\right]+\frac{1}{\varepsilon}\partial_{y_3}\left[\tilde v\right]\right)dx^{\prime}dy \\
&&+\sqrt{2}g\,a_\varepsilon^2\int_{\omega\times Y}\left|\mathbb{D}_{x^\prime}\left[\tilde v\right]+\frac{1}{a_\varepsilon} \mathbb{D}_{y^\prime}\left[\tilde v\right]+\frac{1}{\varepsilon}\partial_{y_3}\left[\tilde v\right]\right|dx'dy-\sqrt{2}g\, a_\varepsilon \int_{\omega\times Y}\left|\frac{1}{a_\varepsilon} \mathbb{D}_{y^\prime}\left[\hat u_\varepsilon\right]+\frac{1}{\varepsilon}\partial_{y_3}\left[\hat u_\varepsilon\right]\right|dx'dy+O_\varepsilon \nonumber \\
&&\ge a_\varepsilon \int_{\omega\times Y}f'\cdot \tilde v'\,dx^\prime dy-\int_{\omega\times Y}f'\cdot \hat u_\varepsilon'\,dx^\prime dy+a_\varepsilon\int_{\omega\times Y} \hat{P}_\varepsilon\, {\rm div}_{x^\prime} \tilde v^\prime\,dx^{\prime}dy+\int_{\omega\times Y} \hat{P}_\varepsilon\, {\rm div}_{y^\prime} \tilde v^\prime\,dx^{\prime}dy \nonumber\\
&&+\frac{a_\varepsilon}{\varepsilon}\int_{\omega\times Y}\hat{P}_\varepsilon\,\partial_{y_3}\tilde v_3\,dx^{\prime}dy+O_\varepsilon.\nonumber
\end{eqnarray}
According with (\ref{convUgorro1}), the first term in relation (\ref{problem2_pressure1}) can be written by the following way
\begin{equation}\label{limit1_pressure1}
2\mu a_\varepsilon\int_{\omega\times Y}\left(\frac{1}{a_\varepsilon^2}\mathbb{D}_{y^\prime} \left[\hat{u}_\varepsilon \right] +{a_\varepsilon \over \varepsilon}\frac{1}{a_\varepsilon^2}\partial_{y_3}\left[\hat{u}_\varepsilon \right]\right):\left(\mathbb{D}_{y^\prime}\left[\tilde v\right]+\frac{a_\varepsilon}{\varepsilon}\partial_{y_3}\left[\tilde v\right]\right)dx^{\prime}dy\to 0,  \text{ as } \varepsilon\to 0.
\end{equation}
In order to pass to the limit in the first nonlinear term, we have
\begin{eqnarray}\label{limit2_pressure1}
&&\sqrt{2}ga_\varepsilon\int_{\omega\times Y}\left|a_\varepsilon \mathbb{D}_{x^\prime}\left[\tilde v\right]+
 \mathbb{D}_{y^\prime}\left[\tilde v\right]+\frac{a_\varepsilon}{\varepsilon}\partial_{y_3}\left[\tilde v\right]\right|dx'dy \to 0, \text{ as } \varepsilon\to 0.
\end{eqnarray}
Now, in order to pass the limit in the second nonlinear term, we are taking into account that
\begin{eqnarray*}
\sqrt{2}g \,a_\varepsilon \int_{\omega\times Y}\left|\frac{1}{a_\varepsilon} \mathbb{D}_{y^\prime}\left[\hat u_\varepsilon\right]+
\frac{1}{\varepsilon}\partial_{y_3}\left[\hat u_\varepsilon\right]\right|dx'dy=\sqrt{2}g\, a_\varepsilon^2\int_{\omega\times Y}\left|\frac{1}{a_\varepsilon^2} \mathbb{D}_{y^\prime}\left[\hat u_\varepsilon\right]+\frac{a_\varepsilon}{\varepsilon}{1\over a_\varepsilon^2}\partial_{y_3}\left[\hat u_\varepsilon\right]\right|dx'dy,
\end{eqnarray*}
and using (\ref{convUgorro1}) and the fact that the function $E(\varphi)=|\varphi|$ is proper convex continuous, we can deduce that
\begin{eqnarray}\label{limit3_pressure1}
\liminf_{\varepsilon \to 0}\sqrt{2}g\, a_\varepsilon \int_{\omega\times Y}\left|\frac{1}{a_\varepsilon} \mathbb{D}_{y^\prime}\left[\hat u_\varepsilon\right]+\frac{1}{\varepsilon}\partial_{y_3}\left[\hat u_\varepsilon\right]\right|dx'dy\ge0.
\end{eqnarray}
Moreover, using (\ref{convUgorro1}) the two first terms in the right hand side of (\ref{problem2_pressure1}) can be written by 
\begin{equation}\label{limit4_pressure1}
a_\varepsilon\int_{\omega\times Y}f'\cdot \tilde v'\,dx'dy-a_\varepsilon^2\int_{\omega\times Y}f'\cdot{\hat u_\varepsilon'\over a_\varepsilon^2}\,dx^\prime dy\to 0, \text{ as } \varepsilon\to 0.
\end{equation}
We consider now the terms which involve the pressure. Taking into account the convergence of the pressure (\ref{convUgorro1}), passing to the limit when $\varepsilon$ tends to zero, we have
\begin{equation}\label{limit5_pressure1}
\int_{\omega\times Y} \hat{P}\, {\rm div}_{\lambda} \tilde v\,dx^{\prime}dy.
\end{equation}
Therefore, taking into account (\ref{limit1_pressure1})-(\ref{limit5_pressure1}), when we pass to the limit in (\ref{problem2_pressure1}) when $\varepsilon$ tends to zero, we have $ 0\ge\int_{\omega\times Y} \hat{P}\, {\rm div}_{\lambda} \tilde v\,dx^{\prime}dy.$
Now, if we choose as test function $-a_\varepsilon\tilde v(x',x'/a_\varepsilon, y_3)$ in (\ref{extension_problem}) and we argue similarly, we obtain $\int_{\omega\times Y} \hat{P}\, {\rm div}_{\lambda} \tilde v\,dx^{\prime}dy\ge0.$
Thus, we can deduce that $\int_{\omega\times Y} \hat{P}\, {\rm div}_{\lambda} \tilde v\,dx^{\prime}dy=0,$ which shows that $\hat P$ does not depend on $y$.
\end{proof}

\begin{theorem}[Critical case]\label{CriticalCase}
If $a_{\varepsilon}\approx \varepsilon$, with $a_\varepsilon/\varepsilon\to \lambda$, $0<\lambda<+\infty$, then $(\hat u_\varepsilon/a_\varepsilon^2,\hat P_\varepsilon)$ converges to $(\hat u,\hat P)$ in $L^2(\omega;H^1(Y)^3)\times L^{2}_0(\omega \times Y)$, which satisfies the following variational inequality 
\begin{eqnarray}\label{limit_critical}
&&2\mu\int_{\omega\times Y}\mathbb{D}_{\lambda} \left[\hat{u} \right] :\left(\mathbb{D}_{\lambda}\left[\tilde v\right] -
\mathbb{D}_{\lambda}\left[\hat u\right] \right)dx^{\prime}dy+ \sqrt{2}g \int_{\omega\times Y}\left| \mathbb{D}_{\lambda}\left[\tilde v\right]\right|dx'dy
- \sqrt{2}g\int_{\omega\times Y}\left| \mathbb{D}_{\lambda}\left[\hat u\right]\right|dx'dy \nonumber\\
&&\ge\int_{\omega\times Y}f'\cdot \left(\tilde v' - \hat u' \right)\,dx^\prime dy-\int_{\omega\times Y}\nabla_{x'}\hat P\,
\left( \tilde v' - \hat u' \right) \,dx^\prime dy,
\end{eqnarray}
where $\mathbb{D}_\lambda[\cdot]=\mathbb{D}_{y^\prime}[\cdot]+\lambda\partial_{y_3}[\cdot]$ and 
for every $\tilde v\in L^2(\omega;H^1(Y)^3)$ such that 
$$\tilde v(x',y)=0 \text{ in } \omega\times Y_s,\quad {\rm div}_{\lambda}\tilde v=0 \text{ in }\omega\times Y, \quad
 \left( \int_Y \tilde v^{\prime}(x^{\prime},y)dy \right)\cdot n=0\text{\ on \ }\partial \omega.$$
\end{theorem}
\begin{proof}
We choose a test function $\tilde v(x^{\prime},y)\in \mathcal{D}(\omega;C_{\sharp}^{\infty}(Y)^3)$ with $\tilde v(x^{\prime},y)=0\in \omega \times Y_s$ (thus, we have that $\tilde v(x^{\prime},x^{\prime}/a_{\varepsilon},y_3)\in (H_0^{1}(\widetilde{\Omega}_{\varepsilon}))^3$). 
We first multiply (\ref{extension_problem}) by $a_\varepsilon^{-2}$ and we use that ${\rm div}_\varepsilon \tilde u_\varepsilon=0$. Then, we take as test function $a_\varepsilon^{2}\tilde v_\varepsilon=a_\varepsilon^{2} (\tilde v'(x',x'/a_\varepsilon, y_3), \lambda  
 \varepsilon/a_\varepsilon v_3(x',x'/a_\varepsilon, y_3))$,
 with $\tilde v(x',y)=0$ in $\omega\times Y_s$ and satisfying the incompressibility  conditions (\ref{divmacro1}),
that is, ${\rm div}_{\lambda}\tilde v=0$ in $\omega\times Y$ and $\left( \int_Y \tilde v^{\prime}(x^{\prime},y)dy \right)\cdot n=0$ on $\partial \omega$, and we have
\begin{eqnarray}\label{extension_problem2}
&&2\mu\int_{\Omega}\mathbb{D}_\varepsilon \left[\tilde{u}_\varepsilon \right] :\left(\mathbb{D}_{x^\prime}\left[\tilde v_\varepsilon\right]+\frac{1}{a_\varepsilon} \mathbb{D}_{y^\prime}\left[\tilde v_\varepsilon\right]+\frac{1}{\varepsilon}\partial_{y_3}\left[\tilde v_\varepsilon\right]\right) dx^{\prime}dy_3-2\mu {1\over a_\varepsilon^2}\int_{\Omega}|\mathbb{D}_\varepsilon \left[\tilde{u}_\varepsilon \right]|^2dx'dy_3 \\
&&+\sqrt{2}g \, a_\varepsilon\int_{\Omega}\left|\mathbb{D}_{x^\prime}\left[\tilde v_\varepsilon\right]+\frac{1}{a_\varepsilon} 
\mathbb{D}_{y^\prime}\left[\tilde v_\varepsilon\right]+\frac{1}{\varepsilon}\partial_{y_3}\left[\tilde 
v_\varepsilon\right]\right|dx'dy_3-\sqrt{2}g{1\over a_\varepsilon}\int_{\Omega}|\mathbb{D}_{\varepsilon}[\tilde u_\varepsilon]|dx'dy_3 \nonumber\\
&&\ge\int_{\Omega}f'\cdot\tilde v'\,dx^{\prime}dy_3-{1\over a_\varepsilon^2}\int_{\Omega}f'\cdot\tilde{u}'_{\varepsilon}\,dx^{\prime}dy_3+\int_{\Omega} \tilde{P}_\varepsilon\, {\rm div}_{x^\prime} \tilde v^\prime\,dx^{\prime}dy_3+\frac{1}{a_\varepsilon}\int_{\Omega} \tilde{P}_\varepsilon\, {\rm div}_{y^\prime} \tilde v^\prime\,dx^{\prime}dy_3 +{\lambda\over a_\varepsilon}\int_{\Omega}\tilde{P}_\varepsilon\,\partial_{y_3}\tilde v_3\,dx^{\prime}dy_3. \nonumber
\end{eqnarray}

By the change of variables given in Remark \ref{remarkCV} and by Lemma \ref{estCV}, we have (\ref{term1_CV}) for the first term in relation (\ref{extension_problem2}), and for the second term in relation (\ref{extension_problem2}) we obtain
\begin{eqnarray}\label{term3_CV}
\int_{\Omega}|\mathbb{D}_\varepsilon \left[\tilde{u}_\varepsilon \right]|^2dx'dy_3=
\int_{\omega\times Y}\left|\frac{1}{a_\varepsilon} \mathbb{D}_{y^\prime}\left[\hat u_\varepsilon\right]+\frac{1}{\varepsilon}
\partial_{y_3}\left[\hat u_\varepsilon\right]\right|^2dx'dy.
\end{eqnarray}
Moreover, applying the change of variables given in Remark \ref{remarkCV} to the fourth term in relation 
(\ref{extension_problem2}), we have (\ref{term2_CV}). Therefore, applying the change of variables given in 
Remark \ref{remarkCV} to relation (\ref{extension_problem2}), we obtain
\begin{eqnarray}\label{formvarcvuprime}
&&2\mu\int_{\omega\times Y}\left(\frac{1}{a_\varepsilon}\mathbb{D}_{y^\prime} \left[\hat{u}_\varepsilon \right] +
\frac{1}{\varepsilon}\partial_{y_3}\left[\hat{u}_\varepsilon \right]\right):\left(\frac{1}{a_\varepsilon} 
\mathbb{D}_{y^\prime}\left[\tilde v_\varepsilon\right]+\frac{1}{\varepsilon}\partial_{y_3}\left[\tilde v_\varepsilon\right]\right)dx^{\prime}dy -2\mu {1\over a_\varepsilon^2}\int_{\omega\times Y}\left|\frac{1}{a_\varepsilon} \mathbb{D}_{y^\prime}\left[\hat u_\varepsilon\right]+\frac{1}{\varepsilon}
\partial_{y_3}\left[\hat u_\varepsilon\right]\right|^2dx'dy\nonumber\\
&&+\sqrt{2}g \, a_\varepsilon\int_{\omega\times Y}\left|\mathbb{D}_{x^\prime}\left[\tilde v_\varepsilon\right]+
\frac{1}{a_\varepsilon} \mathbb{D}_{y^\prime}\left[\tilde v_\varepsilon\right]+\frac{1}{\varepsilon}\partial_{y_3}
\left[\tilde v_\varepsilon\right]\right|dx'dy -\sqrt{2}g{1\over a_\varepsilon}\int_{\omega\times Y}\left|\frac{1}{a_\varepsilon} \mathbb{D}_{y^\prime}\left[\hat u_\varepsilon\right]+\frac{1}{\varepsilon}\partial_{y_3}\left[\hat u_\varepsilon\right]\right|dx'dy+O_\varepsilon \\
&&\ge\int_{\omega\times Y}f'\cdot \tilde v'\,dx^\prime dy-{1\over a_\varepsilon^2}\int_{\omega\times Y}f'\cdot \hat u_\varepsilon'\,dx^\prime dy+\int_{\omega\times Y} \hat{P}_\varepsilon\, {\rm div}_{x^\prime} \tilde v^\prime\,dx^{\prime}dy+\frac{1}{a_\varepsilon}\int_{\omega\times Y} \hat{P}_\varepsilon\, {\rm div}_{y^\prime} \tilde v^\prime\,dx^{\prime}dy \nonumber\\
&&+{\lambda\over a_\varepsilon}\int_{\omega\times Y}\hat{P}_\varepsilon\,\partial_{y_3}\tilde v_3\,dx^{\prime}dy+O_\varepsilon.\nonumber
\end{eqnarray}

According with (\ref{convUgorro1}), the first term in relation (\ref{formvarcvuprime}) can be written 
\begin{equation*}
2\mu\int_{\omega\times Y}\left(\frac{1}{a_\varepsilon^2}\mathbb{D}_{y^\prime} \left[\hat{u}_\varepsilon \right] +{a_\varepsilon \over \varepsilon}\frac{1}{a_\varepsilon^2}\partial_{y_3}\left[\hat{u}_\varepsilon \right]\right):\left(\mathbb{D}_{y^\prime}\left[\tilde v_\varepsilon\right]+\frac{a_\varepsilon}{\varepsilon}\partial_{y_3}\left[\tilde v_\varepsilon\right]\right)dx^{\prime}dy,
\end{equation*}
and, taking into account that $\lambda\,\varepsilon/a_\varepsilon\to 1$, this term tends to the following limit
\begin{equation}\label{limit1}
2\mu\int_{\omega\times Y}\left(\mathbb{D}_{y^\prime} \left[\hat{u} \right] +\lambda\partial_{y_3}\left[\hat{u} \right]\right):\left(\mathbb{D}_{y^\prime}\left[\tilde v\right]+\lambda\partial_{y_3}\left[\tilde v\right]\right)dx^{\prime}dy.
\end{equation}
The second term in relation (\ref{formvarcvuprime}) writes
\begin{equation*}
2\mu\int_{\omega\times Y}\left(\frac{1}{a_\varepsilon^2}\mathbb{D}_{y^\prime} \left[\hat{u}_\varepsilon \right] +
{a_\varepsilon \over \varepsilon}\frac{1}{a_\varepsilon^2}\partial_{y_3}\left[\hat{u}_\varepsilon \right]\right):
 \left(\frac{1}{a_\varepsilon^2}\mathbb{D}_{y^\prime} \left[\hat{u}_\varepsilon \right] +
{a_\varepsilon \over \varepsilon}\frac{1}{a_\varepsilon^2}\partial_{y_3}\left[\hat{u}_\varepsilon \right]\right) dx^{\prime}dy,
\end{equation*}
and, taking into account that the function $B(\varphi) =| \varphi |$ is proper convex continuous and $\lambda\,\varepsilon/a_\varepsilon\to 1$, we get that the $\liminf_{\varepsilon \to 0}$ of  this second is greater or equal than 
\begin{equation}\label{limit2}
 2\mu\int_{\omega\times Y}\left(\mathbb{D}_{y^\prime} \left[\hat{u} \right] +
\lambda\partial_{y_3}\left[\hat{u} \right]\right):\left(\mathbb{D}_{y^\prime}\left[\hat{u} \right]+
\lambda\partial_{y_3}\left[\hat{u} \right]\right)dx^{\prime}dy.
\end{equation}
In order to pass to the limit in the first nonlinear term, we have
\begin{eqnarray*}
&&\left|\sqrt{2}g \,a_\varepsilon \int_{\omega\times Y}\left|\mathbb{D}_{x^\prime}\left[\tilde v_\varepsilon\right]+
\frac{1}{a_\varepsilon} \mathbb{D}_{y^\prime}\left[\tilde v_\varepsilon\right]+\frac{1}{\varepsilon}\partial_{y_3}
\left[\tilde v_\varepsilon\right]\right|dx'dy-\sqrt{2}g \int_{\omega\times Y}\left| \mathbb{D}_{y^\prime}
\left[\tilde v\right]+\lambda\partial_{y_3}\left[\tilde v\right]\right|dx'dy\right|\\
&&\leq \sqrt{2}g\int_{\omega\times Y}\left|a_\varepsilon  \mathbb{D}_{x^\prime}
\left[\tilde v_\varepsilon\right]+ \mathbb{D}_{y^\prime}\left[\tilde v_\varepsilon\right]+\frac{a_\varepsilon}{\varepsilon}
\partial_{y_3}\left[\tilde v_\varepsilon\right]-\mathbb{D}_{y^\prime}\left[\tilde v\right]-\lambda\partial_{y_3}\left[\tilde v\right]\right|dx'dy\\
&&\leq \sqrt{2}g\int_{\omega\times Y}\left|a_\varepsilon  \mathbb{D}_{x^\prime}\left[\tilde v_\varepsilon\right] \right|dx'dy+
\sqrt{2}g\int_{\omega\times Y}\left| \mathbb{D}_{y^\prime}\left[\tilde v_\varepsilon\right]-\mathbb{D}_{y^\prime}\left[\tilde v\right]\right|dx'dy\nonumber\\
&&+\sqrt{2}g\int_{\omega\times Y}\left|\frac{a_\varepsilon}{\varepsilon}\partial_{y_3}\left[\tilde v_\varepsilon\right]-\lambda \partial_{y_3}\left[\tilde v\right]\right|dx'dy \to 0, \text{ as } \varepsilon\to 0,
\end{eqnarray*}
and we can deduce that the first nonlinear term tends to the following limit
\begin{equation}\label{limit3}
\sqrt{2}g \int_{\omega\times Y}\left| \mathbb{D}_{y^\prime}\left[\tilde v\right]+
\lambda\partial_{y_3}\left[\tilde v\right]\right|dx'dy.
\end{equation}
Now, in order to pass the limit in the second nonlinear term, we are taking into account that
\begin{eqnarray*}
\sqrt{2}g{1 \over a_\varepsilon}\int_{\omega\times Y}\left|\frac{1}{a_\varepsilon}
\mathbb{D}_{y^\prime}\left[\hat u_\varepsilon\right]+\frac{1}{\varepsilon}\partial_{y_3}\left[\hat u_\varepsilon\right]\right|dx'dy
=\sqrt{2}g \int_{\omega\times Y}\left| \frac{1}{a_\varepsilon^2} \mathbb{D}_{y^\prime}
\left[{\hat u_\varepsilon}\right]+\frac{a_\varepsilon}{\varepsilon} \frac{1}{a_\varepsilon^2}\partial_{y_3}\left[{\hat u_\varepsilon}\right]\right|dx'dy,
\end{eqnarray*}
and using (\ref{convUgorro1}) and the fact that the function $E(\varphi)=|\varphi|$ is proper convex continuous, we can deduce that
\begin{eqnarray}\label{limit4}
\liminf_{\varepsilon \to 0} \sqrt{2}g{1 \over a_\varepsilon}\int_{\omega\times Y}
\left|\frac{1}{a_\varepsilon} \mathbb{D}_{y^\prime}\left[\hat u_\varepsilon\right]+\frac{1}{\varepsilon}\partial_{y_3}
\left[\hat u_\varepsilon\right]\right|dx'dy\ge 
\sqrt{2}g \int_{\omega\times Y}\left| \mathbb{D}_{y^\prime}\left[\hat {u}\right]+
\lambda\partial_{y_3}\left[\hat {u}\right]\right|dx'dy.
\end{eqnarray}
Moreover, using (\ref{convUgorro1}) the two first terms in the right hand side of (\ref{formvarcvuprime}) tend to the following limit
\begin{equation}\label{limit5}
\int_{\omega\times Y}f'\cdot \left( \tilde v'- \hat{u}' \right )\,dx^\prime dy.
\end{equation}
We consider now the terms which involve the pressure. Taking into account the convergence of the pressure (\ref{convUgorro1}) the first term of the pressure tends
to the following limit
$
\int_{\omega\times Y} \hat{P}\, {\rm div}_{x^\prime} \tilde v^\prime\,dx^{\prime}dy,
$
and using (\ref{divmacro1}) and taking into account that $\hat P$ does not depend on $y$, we have
\begin{eqnarray}\label{limit6}
\int_{\omega\times Y} \hat{P}\, {\rm div}_{x^\prime} \tilde v^\prime\,dx^{\prime}dy&=&\int_{\omega\times Y} \hat{P}\, {\rm div}_{x^\prime} \tilde v^\prime\,dx^{\prime}dy-\int_{\omega}\hat P \left({\rm div}_{x^\prime}\int_{Y}\hat u'dy \right)dx' =-\int_{\omega \times Y}\nabla_{x'}\hat P\,(\tilde v'-\hat u')dx'dy.
\end{eqnarray}
Finally, using that ${\rm div}_{\lambda}\tilde v=0$, we have
\begin{equation}\label{limit7}
\frac{1}{a_\varepsilon}\int_{\omega\times Y} \hat{P}_\varepsilon\, {\rm div}_{y^\prime} \tilde v^\prime\,dx^{\prime}dy+\frac{\lambda}{a_\varepsilon}\int_{\omega\times Y}\hat{P}_\varepsilon\,\partial_{y_3}\tilde v_3\,dx^{\prime}dy=0.
\end{equation}
Therefore, taking into account (\ref{limit1})-(\ref{limit7}),  we have (\ref{limit_critical}).
\end{proof}

\subsection{Subcritical case $a_{\varepsilon}\ll \varepsilon$ ($\lambda=0$)}
We obtain some compactness results about the behavior of the sequences $(\tilde u_\varepsilon, \tilde P_\varepsilon)$ 
and $(\hat u_\varepsilon, \hat P_\varepsilon)$ satisfying the {\it a priori} estimates given in Lemmas \ref{Lemma_estimate}-i) and \ref{estCV}-i), respectively. 
\begin{lemma}[Subcritical case]\label{Convergence_subcritical}
For a subsequence of $\varepsilon$ still denoted by $\varepsilon$, there exist $\tilde{u}\in (L^2(\Omega))^3$,
where $\tilde u_3=0$ and $\tilde{u}=0$ on $y_3=\{0,1\}$, $\hat{u}\in L^2(\Omega;H^1_{\sharp}(Y')^3)$ (``$\sharp$'' denotes $Y'$-periodicity), with $\hat{u}=0$ in $\omega\times Y_s$ and $\hat u=0$ on $y_3=\{0,1\}$ such that $\int_{Y}\hat u(x',y)dy=\int_0^1 \tilde u(x',y_3)dy_3$ with $\int_Y \hat u_3dy=0$ and $\hat u_3$ independent of $y_3$, and $\hat{P}\in L^{2}_0(\omega \times Y)$, independent of $y$, such that
\begin{equation}\label{convtilde2}
{\tilde{u}_{\varepsilon} \over a_\varepsilon^2}\rightharpoonup (\tilde{u}^\prime,0)\text{\ in \ }(L^2(\Omega))^3,
\end{equation}
\begin{equation}\label{convUgorro2}
{\hat{u}_{\varepsilon}\over a_\varepsilon^2}\rightharpoonup \hat{u}\text{\ in \ }L^2(\Omega;H^1(Y')^3),\quad \hat{P}_{\varepsilon}\rightharpoonup \hat{P}\text{\ in \ }L^{2}_0(\omega \times Y),
\end{equation}
\begin{equation}\label{divmacro_tilde2}
{\rm div}_{x^{\prime}}\left( \int_0^1 \tilde{u}^{\prime}(x^{\prime},y_3)dy_3 \right)=0 \text{\ in \ }\omega,\quad \left( \int_0^1 \tilde{u}^{\prime}(x^{\prime},y_3)dy_3 \right)\cdot n=0\text{\ on \ }\partial \omega,
\end{equation}
\begin{equation}\label{divmacro2}
{\rm div}_{y'}\hat{u}'=0 \text{\ in \ }\omega\times Y,\quad {\rm div}_{x^{\prime}}\left( \int_Y \hat{u}^{\prime}(x^{\prime},y)dy \right)=0 \text{\ in \ }\omega,\quad \left( \int_Y \hat{u}^{\prime}(x^{\prime},y)dy \right)\cdot n=0\text{\ on \ }\partial \omega.
\end{equation}
\end{lemma}
\begin{proof}
See Lemmas 5.2, 5.3 and 5.4 in \cite{Anguiano_SuarezGrau} for the proof of (\ref{convtilde2})-(\ref{divmacro2}). 
In order to prove that $\hat P$ does not depend on $y'$ we argue as in the proof of Lemma \ref{Convergence_critical} 
using that $a_{\varepsilon}\ll \varepsilon$, and we obtain $\int_{\omega\times Y} \hat{P}\,
{\rm div}_{y'} \tilde v'\,dx^{\prime}dy=0,$ which shows that $\hat P$ does not depend on $y'$.
Now, in order to prove that $\hat P$ does not depend on $y_3$, setting
$\varepsilon \tilde v=\varepsilon(0, \tilde v_3(x',x'/a_\varepsilon,y_3))$ in (\ref{extension_problem}) (we recall that $g(\varepsilon)= g\, a_\varepsilon)$) and using that ${\rm div}_\varepsilon \tilde u_\varepsilon=0$, we have
\begin{eqnarray}\label{problem_pressure2}
&&2\mu \varepsilon\int_{\Omega}\mathbb{D}_\varepsilon \left[\tilde{u}_\varepsilon \right]:\left(\mathbb{D}_{x^\prime}
\left[\tilde v\right]+\frac{1}{a_\varepsilon} \mathbb{D}_{y^\prime}\left[\tilde v\right]+\frac{1}{\varepsilon}\partial_{y_3}\left[\tilde v\right]\right) dx^{\prime}dy_3-2\mu\int_{\Omega}|\mathbb{D}_\varepsilon \left[\tilde{u}_\varepsilon \right]|^2dx'dy_3 \\
&&+\sqrt{2}ga_\varepsilon \, \varepsilon  \int_{\Omega}\left|\mathbb{D}_{x^\prime}\left[\tilde v\right]+\frac{1}{a_\varepsilon} 
\mathbb{D}_{y^\prime}\left[\tilde v\right]+\frac{1}{\varepsilon}\partial_{y_3}\left[\tilde v\right]\right|dx'dy_3-\sqrt{2}
g a_\varepsilon\int_{\Omega}|\mathbb{D}_{\varepsilon}[\tilde u_\varepsilon]|dx'dy_3 \nonumber\\
&&\ge -\int_{\Omega}f'\cdot\tilde{u}'_{\varepsilon}\,dx^{\prime}dy_3+\int_{\Omega}\tilde{P}_\varepsilon\,\partial_{y_3}\tilde v_3\,dx^{\prime}dy_3. \nonumber
\end{eqnarray}
Applying the change of variables given in Remark \ref{remarkCV} to relation (\ref{problem_pressure2}) and taking into account (\ref{term1_CV})-(\ref{term2_CV}), we obtain
\begin{eqnarray}\label{problem2_pressure2}
&&2\mu \varepsilon\int_{\omega\times Y}\left(\frac{1}{a_\varepsilon}\mathbb{D}_{y^\prime} \left[\hat{u}_\varepsilon \right] +\frac{1}{\varepsilon}\partial_{y_3}\left[\hat{u}_\varepsilon \right]\right):\left(\frac{1}{a_\varepsilon} \mathbb{D}_{y^\prime}\left[\tilde v\right]+\frac{1}{\varepsilon}\partial_{y_3}\left[\tilde v\right]\right)dx^{\prime}dy \\
&&+\sqrt{2}ga_\varepsilon \,\varepsilon\int_{\omega\times Y}\left|\mathbb{D}_{x^\prime}\left[\tilde v\right]+
\frac{1}{a_\varepsilon} \mathbb{D}_{y^\prime}\left[\tilde v\right]+
\frac{1}{\varepsilon}\partial_{y_3}\left[\tilde v\right]\right|dx'dy-\sqrt{2}g a_\varepsilon \int_{\omega\times Y}\left|\frac{1}{a_\varepsilon} \mathbb{D}_{y^\prime}\left[\hat u_\varepsilon\right]+\frac{1}{\varepsilon}\partial_{y_3}\left[\hat u_\varepsilon\right]\right|dx'dy+O_\varepsilon \nonumber \\
&&\ge -\int_{\omega\times Y}f'\cdot \hat u_\varepsilon'\,dx^\prime dy+\int_{\omega\times Y}\hat{P}_\varepsilon\,\partial_{y_3}\tilde v_3\,dx^{\prime}dy+O_\varepsilon.\nonumber
\end{eqnarray}
According with (\ref{convUgorro2}) and using that $a_{\varepsilon}\ll \varepsilon$, the first term in relation (\ref{problem2_pressure2}) can be written by the following way
\begin{equation}\label{limit1_pressure2}
2\mu \varepsilon\int_{\omega\times Y}\left(\frac{1}{a_\varepsilon^2}\mathbb{D}_{y^\prime} \left[\hat{u}_\varepsilon \right] +{a_\varepsilon \over \varepsilon}\frac{1}{a_\varepsilon^2}\partial_{y_3}\left[\hat{u}_\varepsilon \right]\right):\left(\mathbb{D}_{y^\prime}\left[\tilde v\right]+\frac{a_\varepsilon}{\varepsilon}\partial_{y_3}\left[\tilde v\right]\right)dx^{\prime}dy\to 0,  \text{ as } \varepsilon\to 0.
\end{equation}
In order to pass to the limit in the first nonlinear term, we have
\begin{eqnarray}\label{limit2_pressure2}
&&\sqrt{2}g \varepsilon \int_{\omega\times Y}\left|a_\varepsilon \mathbb{D}_{x^\prime}\left[\tilde v\right]+
\mathbb{D}_{y^\prime}\left[\tilde v\right]+\frac{a_\varepsilon}{\varepsilon}\partial_{y_3}\left[\tilde v\right]\right|dx'dy \to 0, \text{ as } \varepsilon\to 0.
\end{eqnarray}
In order to pass to the limit in the second nonlinear term, we proceed as in Lemma \ref{Convergence_critical}.
Moreover, using (\ref{convUgorro2}) the first term in the right hand side of (\ref{problem2_pressure2}) can be written by 
\begin{equation}\label{limit3_pressure2}
a_\varepsilon^2\int_{\omega\times Y}f'\cdot \frac{\hat u_\varepsilon'}{a_\varepsilon^2}\,dx^\prime dy\to 0, \text{ as } \varepsilon\to 0.
\end{equation}
We consider now the term which involves the pressure. Taking into account the convergence of the pressure (\ref{convUgorro2}), passing to the limit when $\varepsilon$ tends to zero, we have
\begin{equation}\label{limit4_pressure2}
\int_{\omega\times Y} \hat{P}\, \partial_{y_3} \tilde v_3\,dx^{\prime}dy.
\end{equation}
Therefore, taking into account (\ref{limit3_pressure1}) and (\ref{limit1_pressure2})-(\ref{limit4_pressure2}), when we pass to the limit in (\ref{problem2_pressure2}) when $\varepsilon$ tends to zero, we have $0\ge\int_{\omega\times Y} \hat{P}\, \partial_{y_3} \tilde v_3\,dx^{\prime}dy.$ Now, if we choose as test function $-\varepsilon \tilde v=-\varepsilon(0, \tilde v_3(x',x'/a_\varepsilon,y_3))$ in (\ref{extension_problem}) and we argue similarly, we can deduce that $\hat P$ does not depend on $y_3$, so $\hat P$ does not depend on $y$.
\end{proof}

\begin{theorem}[Subcritical case]\label{SubCriticalCase}
If $a_{\varepsilon}\ll \varepsilon$, then $(\hat u_\varepsilon/a_\varepsilon^2,\hat P_\varepsilon)$ converges to
$(\hat u,\hat P)$ in $L^2(\Omega;H^1(Y')^3)\times L^{2}_0(\omega \times Y)$, which satisfies the following variational 
inequality
\begin{eqnarray}\label{limit_subcritical}
&&2\mu\int_{\omega\times Y}\mathbb{D}_{y'} \left[\hat{u}' \right] :\left(\mathbb{D}_{y'}\left[\tilde v'\right] -
\mathbb{D}_{y'}\left[\hat u'\right] \right)dx^{\prime}dy+ \sqrt{2}g \int_{\omega\times Y}\left| \mathbb{D}_{y'}\left[\tilde v'\right]\right|dx'dy
- \sqrt{2}g\int_{\omega\times Y}\left| \mathbb{D}_{y'}\left[\hat u'\right]\right|dx'dy \nonumber\\
&&\ge\int_{\omega\times Y}f'\cdot \left(\tilde v' - \hat u' \right)\,dx^\prime dy-\int_{\omega\times Y}\nabla_{x'}\hat P\,
\left( \tilde v' - \hat u' \right) \,dx^\prime dy,
\end{eqnarray}
for every $\tilde v\in L^2(\Omega;H^1(Y')^3)$ such that 
$$\tilde v(x',y)=0 \text{ in } \omega\times Y_s,\quad {\rm div}_{y'}\tilde v'=0 \text{ in }\omega\times Y,\quad
\left( \int_Y \tilde v^{\prime}(x^{\prime},y)dy \right)\cdot n=0\text{\ on \ }\partial \omega.$$
\end{theorem}
\begin{proof}
We choose a test function $\tilde v(x^{\prime},y)\in \mathcal{D}(\omega;C_{\sharp}^{\infty}(Y)^3)$ with $\tilde v(x^{\prime},y)=0\in \omega \times Y_s$ (thus, we have that $\tilde v(x^{\prime},x^{\prime}/a_{\varepsilon},y_3)\in (H_0^{1}(\widetilde{\Omega}_{\varepsilon}))^3$). 
We first multiply (\ref{extension_problem}) by $a_\varepsilon^{-2}$ and we use that ${\rm div}_\varepsilon \tilde u_\varepsilon=0$. Then, we take a test function $a_\varepsilon^{2} \tilde v(x',x'/a_\varepsilon, y_3)$, with $\tilde v_3$ independent of $y_3$ and with $\tilde v(x',y)=0$ in $\omega\times Y_s$ and satisfying the incompressibility  conditions (\ref{divmacro2}), that is, ${\rm div}_{y'}\tilde v'=0$ in $\omega\times Y$ and $\left( \int_Y \tilde v^{\prime}(x^{\prime},y)dy \right)\cdot n=0$ on $\partial \omega$, and we have
\begin{eqnarray}\label{extension_problem2_subcritical}
&&2\mu\int_{\Omega}\mathbb{D}_\varepsilon \left[\tilde{u}_\varepsilon \right] :\left(\mathbb{D}_{x^\prime}\left[\tilde v\right]+\frac{1}{a_\varepsilon} \mathbb{D}_{y^\prime}\left[\tilde v\right]+\frac{1}{\varepsilon}\partial_{y_3}\left[\tilde v\right]\right) dx^{\prime}dy_3-2\mu {1\over a_\varepsilon^2}\int_{\Omega}|\mathbb{D}_\varepsilon \left[\tilde{u}_\varepsilon \right]|^2dx'dy_3 \\
&&+\sqrt{2}g\,a_\varepsilon \int_{\Omega}\left|\mathbb{D}_{x^\prime}\left[\tilde v\right]+
\frac{1}{a_\varepsilon} \mathbb{D}_{y^\prime}\left[\tilde v\right]+\frac{1}{\varepsilon}\partial_{y_3}
\left[\tilde v\right]\right|dx'dy_3-\sqrt{2}g{1\over a_\varepsilon}\int_{\Omega}|\mathbb{D}_{\varepsilon}[\tilde u_\varepsilon]|dx'dy_3 \nonumber\\
&&\ge\int_{\Omega}f'\cdot\tilde v'\,dx^{\prime}dy_3-{1\over a_\varepsilon^2}\int_{\Omega}f'\cdot\tilde{u}'_{\varepsilon}\,dx^{\prime}dy_3+\int_{\Omega} \tilde{P}_\varepsilon\, {\rm div}_{x^\prime} \tilde v^\prime\,dx^{\prime}dy_3+\frac{1}{a_\varepsilon}\int_{\Omega} \tilde{P}_\varepsilon\, {\rm div}_{y^\prime} \tilde v^\prime\,dx^{\prime}dy_3. \nonumber
\end{eqnarray}
Applying the change of variables given in Remark \ref{remarkCV} to relation (\ref{extension_problem2_subcritical}) and taking into account (\ref{term1_CV}), (\ref{term2_CV}) and (\ref{term3_CV}), we obtain
\begin{eqnarray}\label{formvarcvuprime2}
&&2\mu\int_{\omega\times Y}\left(\frac{1}{a_\varepsilon}\mathbb{D}_{y^\prime} \left[\hat{u}_\varepsilon \right] +
\frac{1}{\varepsilon}\partial_{y_3}\left[\hat{u}_\varepsilon \right]\right):\left(\frac{1}{a_\varepsilon} 
\mathbb{D}_{y^\prime}\left[\tilde v\right]+\frac{1}{\varepsilon}\partial_{y_3}\left[\tilde v\right]\right)dx^{\prime}dy \\
&&-2\mu {1\over a_\varepsilon^2}\int_{\omega\times Y}\left|\frac{1}{a_\varepsilon} \mathbb{D}_{y^\prime}\left[\hat u_\varepsilon\right]+\frac{1}{\varepsilon}
\partial_{y_3}\left[\hat u_\varepsilon\right]\right|^2dx'dy+\sqrt{2}g \,a_\varepsilon\int_{\omega\times Y}\left|\mathbb{D}_{x^\prime}\left[\tilde v\right]+\frac{1}{a_\varepsilon} \mathbb{D}_{y^\prime}\left[\tilde v\right]+\frac{1}{\varepsilon}\partial_{y_3}\left[\tilde v\right]\right|dx'dy \nonumber\\
&&-\sqrt{2}g{1\over a_\varepsilon}\int_{\omega\times Y}\left|\frac{1}{a_\varepsilon} \mathbb{D}_{y^\prime}\left[\hat u_\varepsilon\right]+\frac{1}{\varepsilon}\partial_{y_3}\left[\hat u_\varepsilon\right]\right|dx'dy+O_\varepsilon \nonumber \\
&&\ge\int_{\omega\times Y}f'\cdot \tilde v'\,dx^\prime dy-{1\over a_\varepsilon^2}\int_{\omega\times Y}f'\cdot \hat u_\varepsilon'\,dx^\prime dy+\int_{\omega\times Y} \hat{P}_\varepsilon\, {\rm div}_{x^\prime} \tilde v^\prime\,dx^{\prime}dy+\frac{1}{a_\varepsilon}\int_{\omega\times Y} \hat{P}_\varepsilon\, {\rm div}_{y^\prime} \tilde v^\prime\,dx^{\prime}dy+O_\varepsilon. \nonumber
\end{eqnarray}
In the left-hand side, we only  give the details of convergence for the first nonlinear term, the most challenging one. 
\begin{eqnarray*}
&&\left|\sqrt{2}g \,a_\varepsilon \int_{\omega\times Y}\left|\mathbb{D}_{x^\prime}\left[\tilde v\right]+
\frac{1}{a_\varepsilon} \mathbb{D}_{y^\prime}\left[\tilde v\right]+\frac{1}{\varepsilon}\partial_{y_3}
\left[\tilde v\right]\right|dx'dy-\sqrt{2}g \int_{\omega\times Y}\left| \mathbb{D}_{y^\prime}
\left[\tilde v\right]\right|dx'dy \right|\\
&&\leq \sqrt{2}g\int_{\omega\times Y}\left|a_\varepsilon  \mathbb{D}_{x^\prime}
\left[\tilde v\right]+ \mathbb{D}_{y^\prime}\left[\tilde v\right]+\frac{a_\varepsilon}{\varepsilon}
\partial_{y_3}\left[\tilde v\right]-\mathbb{D}_{y^\prime}\left[\tilde v\right]\right|dx'dy\\
&&\leq \sqrt{2}g\int_{\omega\times Y}\left|a_\varepsilon  \mathbb{D}_{x^\prime}\left[\tilde v\right] \right|dx'dy+\sqrt{2}g\int_{\omega\times Y}\left|\frac{a_\varepsilon}{\varepsilon}\partial_{y_3}\left[\tilde v\right]
\right|dx'dy \to 0, \text{ as } \varepsilon\to 0.
\end{eqnarray*}

Using (\ref{convUgorro2}) the two first terms in the right hand side of (\ref{formvarcvuprime2}) tend to the following limit
\begin{equation*}\label{limit5_case2}
\int_{\omega\times Y}f'\cdot (\tilde v'-\hat u')\,dx^\prime dy.
\end{equation*}
We consider now the terms which involve the pressure. Taking into account the convergence of the pressure (\ref{convUgorro2})
the first term of the pressure tends to the following limit
$
\int_{\omega\times Y} \hat{P}\, {\rm div}_{x^\prime} \tilde v^\prime\,dx^{\prime}dy,
$
and using (\ref{divmacro2}) and taking into account that $\hat P$ does not depend on $y$, we have (\ref{limit6}).
Finally, using that ${\rm div}_{y'}\tilde v'=0$, we have
\begin{equation}\label{limit7_case2}
\frac{1}{a_\varepsilon}\int_{\omega\times Y} \hat{P}_\varepsilon\, {\rm div}_{y^\prime} \tilde v^\prime\,dx^{\prime}dy=0.
\end{equation}
It is straightforward to obtain that $\hat u_3=0$ and therefore we get 
 (\ref{limit_subcritical}).
\end{proof}

\subsection{Supercritical case $a_{\varepsilon}\gg \varepsilon$ ($\lambda=+\infty$)}
We obtain some compactness results about the behavior of the sequences $(\tilde u_\varepsilon, \tilde P_\varepsilon)$ and $(\hat u_\varepsilon, \hat P_\varepsilon)$ satisfying the {\it a priori} estimates given in Lemmas \ref{Lemma_estimate}-ii) and \ref{estCV}-ii), respectively. 
\begin{lemma}[Supercritical case]\label{Convergence_supercritical}
For a subsequence of $\varepsilon$ still denote by $\varepsilon$, there exist $\tilde{u}\in H^1(0,1;L^2(\omega)^3)$,
where $\tilde u_3=0$ and $\tilde{u}=0$ on $y_3=\{0,1\}$, $\hat{u}\in H^1(0,1;L^2_{\sharp}(\omega\times Y')^3)$ (``$\sharp$'' denotes $Y'$-periodicity), 
with $\hat{u}=0$ in $\omega\times Y_s$, $\hat u=0$ on $y_3=\{0,1\}$ such that $\int_{Y}\hat u(x',y)dy=\int_0^1 \tilde u(x',y_3)dy_3$ with $\int_Y \hat u_3dy=0$ and $\hat u_3$ independent of $y_3$, and $\hat{P}\in L^{2}_0(\omega \times Y)$, independent of $y$, such that
\begin{equation}\label{convtilde3}
{\tilde{u}_{\varepsilon} \over \varepsilon^2}\rightharpoonup (\tilde{u}^\prime,0)\text{\ in \ }H^1(0,1;L^2(\omega)^3),
\end{equation}
\begin{equation}\label{convUgorro3}
{\hat{u}_{\varepsilon}\over \varepsilon^2}\rightharpoonup \hat{u}\text{\ in \ }H^1(0,1;L^2(\omega \times Y')^3),\quad \hat{P}_{\varepsilon}\rightharpoonup \hat{P}\text{\ in \ }L^{2}_0(\omega \times Y),
\end{equation}
\begin{equation}\label{divmacro_tilde3}
{\rm div}_{x^{\prime}}\left( \int_0^1 \tilde{u}^{\prime}(x^{\prime},y_3)dy_3 \right)=0 \text{\ in \ }\omega,\quad \left( \int_0^1 \tilde{u}^{\prime}(x^{\prime},y_3)dy_3 \right)\cdot n=0\text{\ on \ }\partial \omega,
\end{equation}
\begin{equation}\label{divmacro3}
{\rm div}_{y'}\hat{u}'=0 \text{\ in \ }\omega\times Y,\quad {\rm div}_{x^{\prime}}\left( \int_Y \hat{u}^{\prime}(x^{\prime},y)dy \right)=0 \text{\ in \ }\omega,\quad \left( \int_Y \hat{u}^{\prime}(x^{\prime},y)dy \right)\cdot n=0\text{\ on \ }\partial \omega.
\end{equation}
\end{lemma}
\begin{proof}
See Lemmas 5.2, 5.3 and 5.4 in \cite{Anguiano_SuarezGrau} for the proof of (\ref{convtilde3})-(\ref{divmacro3}).
Here, we prove that $\hat P$ does not depend on the microscopic variable $y$. To do this, we choose as test 
function $\tilde v(x^{\prime},y)\in \mathcal{D}(\omega;C_{\sharp}^{\infty}(Y)^3)$ with $\tilde v(x^{\prime},y)
=0\in \omega \times Y_s$ (thus, $\tilde v(x^{\prime},x^{\prime}/a_{\varepsilon},y_3)\in (H_0^{1}(\widetilde{\Omega}_{\varepsilon}))^3$).
In order to prove that $\hat P$ does not depend on $y_3$, we set $\varepsilon\tilde v(x',x'/a_\varepsilon, y_3)$ in
(\ref{extension_problem}) (we recall that $g(\varepsilon)= g\,\varepsilon)$)and using that ${\rm div}_\varepsilon \tilde u_\varepsilon=0$, we have
\begin{eqnarray}\label{problem_pressure3}
&&2\mu \varepsilon\int_{\Omega}\mathbb{D}_\varepsilon \left[\tilde{u}_\varepsilon \right] :\left(\mathbb{D}_{x^\prime}\left[\tilde v\right]+\frac{1}{a_\varepsilon} \mathbb{D}_{y^\prime}\left[\tilde v\right]+\frac{1}{\varepsilon}\partial_{y_3}\left[\tilde v\right]\right) dx^{\prime}dy_3-2\mu\int_{\Omega}|\mathbb{D}_\varepsilon \left[\tilde{u}_\varepsilon \right]|^2dx'dy_3 \\
&&+\sqrt{2}g \varepsilon^2\int_{\Omega}\left|\mathbb{D}_{x^\prime}\left[\tilde v\right]+
\frac{1}{a_\varepsilon} \mathbb{D}_{y^\prime}\left[\tilde v\right]+\frac{1}{\varepsilon}\partial_{y_3}\left[\tilde v\right]\right|dx'dy_3-
\sqrt{2}g \varepsilon \int_{\Omega}|\mathbb{D}_{\varepsilon}[\tilde u_\varepsilon]|dx'dy_3 \nonumber\\
&&\ge \varepsilon \int_{\Omega}f'\cdot\tilde v'\,dx^{\prime}dy_3-\int_{\Omega}f'\cdot\tilde{u}'_{\varepsilon}\,dx^{\prime}dy_3+\varepsilon\int_{\Omega} \tilde{P}_\varepsilon\, {\rm div}_{x^\prime} \tilde v^\prime\,dx^{\prime}dy_3+{\varepsilon \over a_\varepsilon}\int_{\Omega} \tilde{P}_\varepsilon\, {\rm div}_{y^\prime} \tilde v^\prime\,dx^{\prime}dy_3 +\int_{\Omega}\tilde{P}_\varepsilon\,\partial_{y_3}\tilde v_3\,dx^{\prime}dy_3. \nonumber
\end{eqnarray}
Applying the change of variables given in Remark \ref{remarkCV} to relation (\ref{problem_pressure3}) and taking into account (\ref{term1_CV})-(\ref{term2_CV}), we obtain
\begin{eqnarray}\label{problem2_pressure3}
&&2\mu \varepsilon\int_{\omega\times Y}\left(\frac{1}{a_\varepsilon}\mathbb{D}_{y^\prime} \left[\hat{u}_\varepsilon \right] +\frac{1}{\varepsilon}\partial_{y_3}\left[\hat{u}_\varepsilon \right]\right):\left(\frac{1}{a_\varepsilon} \mathbb{D}_{y^\prime}\left[\tilde v\right]+\frac{1}{\varepsilon}\partial_{y_3}\left[\tilde v\right]\right)dx^{\prime}dy \\
&&+\sqrt{2}g \varepsilon^2\int_{\omega\times Y}\left|\mathbb{D}_{x^\prime}\left[\tilde v\right]+\frac{1}{a_\varepsilon} 
\mathbb{D}_{y^\prime}\left[\tilde v\right]+\frac{1}{\varepsilon}\partial_{y_3}\left[\tilde v\right]\right|dx'dy-
\sqrt{2}g \varepsilon \int_{\omega\times Y}\left|\frac{1}{a_\varepsilon} \mathbb{D}_{y^\prime}\left[\hat u_\varepsilon\right]+\frac{1}{\varepsilon}\partial_{y_3}\left[\hat u_\varepsilon\right]\right|dx'dy+O_\varepsilon \nonumber \\
&&\ge \varepsilon \int_{\omega\times Y}f'\cdot \tilde v'\,dx^\prime dy-\int_{\omega\times Y}f'\cdot \hat u_\varepsilon'\,dx^\prime dy+\varepsilon\int_{\omega\times Y} \hat{P}_\varepsilon\, {\rm div}_{x^\prime} \tilde v^\prime\,dx^{\prime}dy+{\varepsilon \over a_\varepsilon}\int_{\omega\times Y} \hat{P}_\varepsilon\, {\rm div}_{y^\prime} \tilde v^\prime\,dx^{\prime}dy \nonumber\\
&&+\int_{\omega\times Y}\hat{P}_\varepsilon\,\partial_{y_3}\tilde v_3\,dx^{\prime}dy+O_\varepsilon.\nonumber
\end{eqnarray}
According with (\ref{convUgorro3}) and using that $a_{\varepsilon}\gg \varepsilon$, one has for the first term in relation (\ref{problem2_pressure3}) 

\begin{equation}\label{limit1_pressure3}
2\mu \varepsilon\int_{\omega\times Y}\left({\varepsilon \over a_\varepsilon}\frac{1}{\varepsilon^2}\mathbb{D}_{y^\prime} \left[\hat{u}_\varepsilon \right] +{1 \over \varepsilon^2}\partial_{y_3}\left[\hat{u}_\varepsilon \right]\right):\left({\varepsilon \over a_\varepsilon}\mathbb{D}_{y^\prime}\left[\tilde v\right]+\partial_{y_3}\left[\tilde v\right]\right)dx^{\prime}dy\to 0,  \text{ as } \varepsilon\to 0.
\end{equation}
We pass to the limit in the first nonlinear term and  we have
\begin{eqnarray}\label{limit2_pressure3}
&&\sqrt{2}g\varepsilon\int_{\omega\times Y}\left|\varepsilon \mathbb{D}_{x^\prime}\left[\tilde v\right]+
{\varepsilon \over a_\varepsilon} \mathbb{D}_{y^\prime}\left[\tilde v\right]+
\partial_{y_3}\left[\tilde v\right]\right|dx'dy \to 0, \text{ as } \varepsilon\to 0.
\end{eqnarray}
In order to pass the limit in the second nonlinear term, we taking into account that
\begin{eqnarray*}
\sqrt{2}g \varepsilon \int_{\omega\times Y}\left|\frac{1}{a_\varepsilon} \mathbb{D}_{y^\prime}\left[\hat u_\varepsilon\right]+
\frac{1}{\varepsilon}\partial_{y_3}\left[\hat u_\varepsilon\right]\right|dx'dy=\sqrt{2}g\varepsilon^2 \int_{\omega\times Y}\left|{\varepsilon \over a_\varepsilon}\frac{1}{\varepsilon^2} \mathbb{D}_{y^\prime}\left[\hat u_\varepsilon\right]+{1\over \varepsilon^2}\partial_{y_3}\left[\hat u_\varepsilon\right]\right|dx'dy,
\end{eqnarray*}
and using (\ref{convUgorro3}), with $a_{\varepsilon}\gg \varepsilon$, and the fact that the function $E(\varphi)=|\varphi|$ is proper convex continuous, we can deduce that
\begin{eqnarray}\label{limit3_pressure3}
\liminf_{\varepsilon \to 0}\sqrt{2}g \varepsilon \int_{\omega\times Y}\left|\frac{1}{a_\varepsilon} \mathbb{D}_{y^\prime}\left[\hat u_\varepsilon\right]+\frac{1}{\varepsilon}\partial_{y_3}\left[\hat u_\varepsilon\right]\right|dx'dy\ge0.
\end{eqnarray}
Moreover, using (\ref{convUgorro3}) the two first terms in the right hand side of (\ref{problem2_pressure3}) can be written by 
\begin{equation}\label{limit4_pressure3}
\varepsilon\int_{\omega\times Y}f'\cdot \tilde v'\,dx'dy-\varepsilon^2\int_{\omega\times Y}f'\cdot{\hat u_\varepsilon'\over \varepsilon^2}\,dx^\prime dy\to 0, \text{ as } \varepsilon\to 0.
\end{equation}
We consider now the terms which involve the pressure. Taking into account the convergence of the pressure (\ref{convUgorro3}) and $a_{\varepsilon}\gg \varepsilon$, passing to the limit when $\varepsilon$ tends to zero, we have
\begin{equation}\label{limit5_pressure3}
\int_{\omega\times Y} \hat{P}\, \partial_{y_3} \tilde v_3\,dx^{\prime}dy.
\end{equation}
Therefore, taking into account (\ref{limit1_pressure3})-(\ref{limit5_pressure3}), when we pass to the limit in (\ref{problem2_pressure3}) when $\varepsilon$ tends to zero, we have
$
0\ge\int_{\omega\times Y} \hat{P}\, \partial_{y_3} \tilde v_3\,dx^{\prime}dy.
$
Now, if we choose as test function $-\varepsilon\tilde v(x',x'/a_\varepsilon, y_3)$ in (\ref{extension_problem}) and we argue similarly, we can deduce that $\hat P$ does not depend on $y_3$.

Now, in order to prove that $\hat P$ does not depend on $y'$, we set $a_\varepsilon \tilde v=a_\varepsilon(\tilde v'(x',x'/a_\varepsilon, y_3),0)$ in (\ref{extension_problem}) and using that ${\rm div}_\varepsilon \tilde u_\varepsilon=0$, we have
\begin{eqnarray}\label{problem_pressure3_1}
&&2\mu a_\varepsilon\int_{\Omega}\mathbb{D}_\varepsilon \left[\tilde{u}_\varepsilon \right] :\left(\mathbb{D}_{x^\prime}\left[\tilde v\right]+\frac{1}{a_\varepsilon} \mathbb{D}_{y^\prime}\left[\tilde v\right]+\frac{1}{\varepsilon}\partial_{y_3}\left[\tilde v\right]\right) dx^{\prime}dy_3-2\mu\int_{\Omega}|\mathbb{D}_\varepsilon \left[\tilde{u}_\varepsilon \right]|^2dx'dy_3 \\
&&+\sqrt{2}g\varepsilon \, a_\varepsilon \int_{\Omega}\left|\mathbb{D}_{x^\prime}\left[\tilde v\right]+\frac{1}{a_\varepsilon} 
\mathbb{D}_{y^\prime}\left[\tilde v\right]+\frac{1}{\varepsilon}\partial_{y_3}\left[\tilde v\right]\right|dx'dy_3-
\sqrt{2}g \varepsilon \int_{\Omega}|\mathbb{D}_{\varepsilon}[\tilde u_\varepsilon]|dx'dy_3 \nonumber\\
&&\ge a_\varepsilon \int_{\Omega}f'\cdot\tilde v'\,dx^{\prime}dy_3-\int_{\Omega}f'\cdot\tilde{u}'_{\varepsilon}\,dx^{\prime}dy_3+a_\varepsilon\int_{\Omega} \tilde{P}_\varepsilon\, {\rm div}_{x^\prime} \tilde v^\prime\,dx^{\prime}dy_3+\int_{\Omega} \tilde{P}_\varepsilon\, {\rm div}_{y^\prime} \tilde v^\prime\,dx^{\prime}dy_3. \nonumber
\end{eqnarray}
Applying the change of variables given in Remark \ref{remarkCV} to relation (\ref{problem_pressure3_1}) and taking into account (\ref{term1_CV})-(\ref{term2_CV}), we obtain
\begin{eqnarray}\label{problem2_pressure3_1}
&&2\mu a_\varepsilon\int_{\omega\times Y}\left(\frac{1}{a_\varepsilon}\mathbb{D}_{y^\prime} \left[\hat{u}_\varepsilon \right] +\frac{1}{\varepsilon}\partial_{y_3}\left[\hat{u}_\varepsilon \right]\right):\left(\frac{1}{a_\varepsilon} \mathbb{D}_{y^\prime}\left[\tilde v\right]+\frac{1}{\varepsilon}\partial_{y_3}\left[\tilde v\right]\right)dx^{\prime}dy \\
&&+\sqrt{2}g\varepsilon\, a_\varepsilon\int_{\omega\times Y}\left|\mathbb{D}_{x^\prime}\left[\tilde v\right]+
\frac{1}{a_\varepsilon} \mathbb{D}_{y^\prime}\left[\tilde v\right]+\frac{1}{\varepsilon}\partial_{y_3}\left[\tilde v\right]\right|dx'dy-
\sqrt{2}g \varepsilon \int_{\omega\times Y}\left|\frac{1}{a_\varepsilon} \mathbb{D}_{y^\prime}\left[\hat u_\varepsilon\right]+\frac{1}{\varepsilon}\partial_{y_3}\left[\hat u_\varepsilon\right]\right|dx'dy+O_\varepsilon \nonumber \\
&&\ge a_\varepsilon \int_{\omega\times Y}f'\cdot \tilde v'\,dx^\prime dy-\int_{\omega\times Y}f'\cdot \hat u_\varepsilon'\,dx^\prime dy+a_\varepsilon\int_{\omega\times Y} \hat{P}_\varepsilon\, {\rm div}_{x^\prime} \tilde v^\prime\,dx^{\prime}dy+\int_{\omega\times Y} \hat{P}_\varepsilon\, {\rm div}_{y^\prime} \tilde v^\prime\,dx^{\prime}dy. \nonumber
\end{eqnarray}
According with (\ref{convUgorro3}) and using that $a_{\varepsilon}\gg \varepsilon$, the first term in relation (\ref{problem2_pressure3_1}) can be written by the following way
\begin{equation}\label{limit1_pressure3_1}
2\mu a_\varepsilon\int_{\omega\times Y}\left({\varepsilon \over a_\varepsilon}\frac{1}{\varepsilon^2}\mathbb{D}_{y^\prime} \left[\hat{u}_\varepsilon \right] +{1 \over \varepsilon^2}\partial_{y_3}\left[\hat{u}_\varepsilon \right]\right):\left({\varepsilon \over a_\varepsilon}\mathbb{D}_{y^\prime}\left[\tilde v\right]+\partial_{y_3}\left[\tilde v\right]\right)dx^{\prime}dy\to 0,  \text{ as } \varepsilon\to 0.
\end{equation}
In order to pass to the limit in the first nonlinear term, we have
\begin{eqnarray}\label{limit2_pressure3_1}
&&\sqrt{2}g a_\varepsilon\int_{\omega\times Y}\left|\varepsilon\mathbb{D}_{x^\prime}\left[\tilde v\right]+{\varepsilon \over a_\varepsilon} \mathbb{D}_{y^\prime}\left[\tilde v\right]+\partial_{y_3}\left[\tilde v\right]\right|dx'dy \to 0, \text{ as } \varepsilon\to 0.
\end{eqnarray}
Moreover, using (\ref{convUgorro3}) the two first terms in the right hand side of (\ref{problem2_pressure3_1}) can be written by 
\begin{equation}\label{limit4_pressure3_1}
a_\varepsilon\int_{\omega\times Y}f'\cdot \tilde v'\,dx'dy-\varepsilon^2\int_{\omega\times Y}f'\cdot{\hat u_\varepsilon'\over \varepsilon^2}\,dx^\prime dy\to 0, \text{ as } \varepsilon\to 0.
\end{equation}
We consider now the terms which involve the pressure. Taking into account the convergence of the pressure (\ref{convUgorro3}), passing to the limit when $\varepsilon$ tends to zero, we have
\begin{equation}\label{limit5_pressure3_1}
\int_{\omega\times Y} \hat{P}\, {\rm div}_{y'} \tilde v'\,dx^{\prime}dy.
\end{equation}
Therefore, taking into account (\ref{limit3_pressure3}) and (\ref{limit1_pressure3_1})-(\ref{limit5_pressure3_1}), when we pass to the limit in (\ref{problem2_pressure3_1}) when $\varepsilon$ tends to zero, we have
$
0\ge\int_{\omega\times Y} \hat{P}\, {\rm div}_{y'} \tilde v'\,dx^{\prime}dy.
$
Now, if we choose as test function $-a_\varepsilon \tilde v=-a_\varepsilon(\tilde v'(x',x'/a_\varepsilon, y_3),0)$ in (\ref{extension_problem}) and we argue similarly, we can deduce that $\hat P$ does not depend on $y'$, so $\hat P$ does not depend on $y$.
\end{proof}

\begin{theorem}[Supercritical case]\label{SupercriticalCase}
If $a_{\varepsilon}\gg \varepsilon$, then $(\hat u_\varepsilon/\varepsilon^2,\hat
P_\varepsilon)$ converges to $(\hat u,\hat P)$ in $H^1(0,1;L^2(\omega \times Y')^3)\times L^{2}_0(\omega \times Y)$, 
which satisfies the following variational equality 
\begin{eqnarray}\label{limit_supercritical}
&&2\mu\int_{\omega\times Y}\partial_{y_3} \left[\hat{u}' \right] :\left(\partial_{y_3} \left[\tilde {v}'\right] -
\partial_{y_3} \left[\hat {u}'\right] \right)dx^{\prime}dy+ \sqrt{2}g \int_{\omega\times Y}\left| 
\partial_{y_3}\left[\tilde {v}'\right]\right|dx'dy
- \sqrt{2}g\int_{\omega\times Y}\left| \partial_{y_3} \left[\hat {u}'\right]\right|dx'dy \nonumber\\
&&\ge\int_{\omega\times Y}f'\cdot \left(\tilde v' - \hat u' \right)\,dx^\prime dy-\int_{\omega\times Y}\nabla_{x'}\hat P\,
\left( \tilde v' - \hat u' \right) \,dx^\prime dy,
\end{eqnarray}
for every $\tilde v\in H^1(0,1;L^2(\omega\times Y')^3)$ such that 
$$\tilde v(x',y)=0 \text{ in } \omega\times Y_s,\quad {\rm div}_{y'}\tilde v'=0 \text{ in }\omega\times Y,\quad
 \left( \int_Y \tilde v^{\prime}(x^{\prime},y)dy \right)\cdot n=0\text{\ on \ }\partial \omega.$$
\end{theorem}
\begin{proof}
We choose a test function $\tilde v(x^{\prime},y)\in \mathcal{D}(\omega;C_{\sharp}^{\infty}(Y)^3)$ with $\tilde v(x^{\prime},y)=0\in \omega \times Y_s$ (thus, $\tilde v(x^{\prime},x^{\prime}/a_{\varepsilon},y_3)\in (H_0^{1}(\widetilde{\Omega}_{\varepsilon}))^3$). 
We first multiply (\ref{extension_problem}) by $\varepsilon^{-2}$ and we use that ${\rm div}_\varepsilon \tilde u_\varepsilon=0$. Then, we take a test function $\varepsilon^{2} \tilde v(x',x'/a_\varepsilon, y_3)$, with $\tilde v_3$ independent of $y_3$ and with $\tilde v(x',y)=0$ in $\omega\times Y_s$ and satisfying the incompressibility  conditions (\ref{divmacro3}), that is, ${\rm div}_{y'}\tilde v'=0$ in $\omega\times Y$ and $\left( \int_Y \tilde v^{\prime}(x^{\prime},y)dy \right)\cdot n=0$ on $\partial \omega$, and we have
\begin{eqnarray}\label{extension_problem2_supercritical}
&&2\mu\int_{\Omega}\mathbb{D}_\varepsilon \left[\tilde{u}_\varepsilon \right] :\left(\mathbb{D}_{x^\prime}\left[\tilde v\right]+\frac{1}{a_\varepsilon} \mathbb{D}_{y^\prime}\left[\tilde v\right]+\frac{1}{\varepsilon}\partial_{y_3}\left[\tilde v\right]\right) dx^{\prime}dy_3-2\mu {1\over \varepsilon^2}\int_{\Omega}|\mathbb{D}_\varepsilon \left[\tilde{u}_\varepsilon \right]|^2dx'dy_3 \\
&&+\sqrt{2}g \varepsilon \int_{\Omega}\left|\mathbb{D}_{x^\prime}\left[\tilde v\right]+
\frac{1}{a_\varepsilon} \mathbb{D}_{y^\prime}\left[\tilde v\right]+\frac{1}{\varepsilon}\partial_{y_3}\left[\tilde v\right]\right|dx'dy_3-
\sqrt{2}g{1\over \varepsilon}\int_{\Omega}|\mathbb{D}_{\varepsilon}[\tilde u_\varepsilon]|dx'dy_3 \nonumber\\
&&\ge\int_{\Omega}f'\cdot\tilde v'\,dx^{\prime}dy_3-{1\over \varepsilon^2}\int_{\Omega}f'\cdot\tilde{u}'_{\varepsilon}\,dx^{\prime}dy_3+\int_{\Omega} \tilde{P}_\varepsilon\, {\rm div}_{x^\prime} \tilde v^\prime\,dx^{\prime}dy_3+\frac{1}{a_\varepsilon}\int_{\Omega} \tilde{P}_\varepsilon\, {\rm div}_{y^\prime} \tilde v^\prime\,dx^{\prime}dy_3. \nonumber
\end{eqnarray}
Applying the change of variables given in Remark \ref{remarkCV} to relation (\ref{extension_problem2_supercritical}), arguing as in the critical case, we obtain
\begin{eqnarray}\label{formvarcvuprime3}
&&2\mu\int_{\omega\times Y}\left(\frac{1}{a_\varepsilon}\mathbb{D}_{y^\prime} \left[\hat{u}_\varepsilon \right] +
\frac{1}{\varepsilon}\partial_{y_3}\left[\hat{u}_\varepsilon \right]\right):\left(\frac{1}{a_\varepsilon} 
\mathbb{D}_{y^\prime}\left[\tilde v\right]+\frac{1}{\varepsilon}\partial_{y_3}\left[\tilde v\right]\right)dx^{\prime}dy \\
&&-2\mu {1\over \varepsilon^2}\int_{\omega\times Y}\left|\frac{1}{a_\varepsilon} \mathbb{D}_{y^\prime}\left[\hat u_\varepsilon\right]+\frac{1}{\varepsilon}
\partial_{y_3}\left[\hat u_\varepsilon\right]\right|^2dx'dy+\sqrt{2}g \varepsilon \int_{\omega\times Y}\left|\mathbb{D}_{x^\prime}\left[\tilde v\right]+\frac{1}{a_\varepsilon} \mathbb{D}_{y^\prime}\left[\tilde v\right]+\frac{1}{\varepsilon}\partial_{y_3}\left[\tilde v\right]\right|dx'dy \nonumber\\
&&-\sqrt{2}g{1\over \varepsilon}\int_{\omega\times Y}\left|\frac{1}{a_\varepsilon} \mathbb{D}_{y^\prime}\left[\hat u_\varepsilon\right]+\frac{1}{\varepsilon}\partial_{y_3}\left[\hat u_\varepsilon\right]\right|dx'dy+O_\varepsilon \nonumber \\
&&\ge\int_{\omega\times Y}f'\cdot \tilde v'\,dx^\prime dy-{1\over \varepsilon^2}\int_{\omega\times Y}f'\cdot \hat u_\varepsilon'\,dx^\prime dy+\int_{\omega\times Y} \hat{P}_\varepsilon\, {\rm div}_{x^\prime} \tilde v^\prime\,dx^{\prime}dy+\frac{1}{a_\varepsilon}\int_{\omega\times Y} \hat{P}_\varepsilon\, {\rm div}_{y^\prime} \tilde v^\prime\,dx^{\prime}dy+O_\varepsilon. \nonumber
\end{eqnarray}
According with (\ref{convUgorro3}), the first term in relation (\ref{formvarcvuprime3}) can be written by the following way
\begin{equation*}
2\mu\int_{\omega\times Y}\left({\varepsilon \over a_\varepsilon}\frac{1}{\varepsilon^2}\mathbb{D}_{y^\prime} \left[\hat{u}_\varepsilon \right] +\frac{1}{\varepsilon^2}\partial_{y_3}\left[\hat{u}_\varepsilon \right]\right):\left({\varepsilon \over a_\varepsilon}\mathbb{D}_{y^\prime}\left[\tilde v\right]+\partial_{y_3}\left[\tilde v\right]\right)dx^{\prime}dy,
\end{equation*}
and, taking into account that $a_\varepsilon\gg \varepsilon$, this term tends to the following limit
\begin{equation}\label{limit1_case3}
2\mu\int_{\omega\times Y}\partial_{y_3}\left[\hat{u}' \right]:\partial_{y_3}\left[\tilde v'\right]dx^{\prime}dy.
\end{equation}
The second term in relation (\ref{formvarcvuprime3}) writes
\begin{equation*}
2\mu\int_{\omega\times Y}\left({\varepsilon \over a_\varepsilon}{1 \over \varepsilon^2}\mathbb{D}_{y^\prime} \left[\hat{u}_\varepsilon \right] +
{1 \over \varepsilon^2}\partial_{y_3}\left[\hat{u}_\varepsilon \right]\right):
 \left({\varepsilon \over a_\varepsilon}{1 \over \varepsilon^2}\mathbb{D}_{y^\prime} \left[\hat{u}_\varepsilon \right] +
{1 \over \varepsilon^2}\partial_{y_3}\left[\hat{u}_\varepsilon \right]\right) dx^{\prime}dy,
\end{equation*}
and, taking into account that the function $B(\varphi) =| \varphi |$ is proper convex continuous and $a_\varepsilon\gg \varepsilon$, we get that the $\liminf_{\varepsilon \to 0}$ of  this second is greater or equal than 
\begin{equation}\label{limit2_case3}
 2\mu\int_{\omega\times Y}\partial_{y_3} \left[\hat{u}' \right] :\partial_{y_3}\left[\hat{u}' \right]dx^{\prime}dy.
\end{equation}
In order to pass to the limit in the first nonlinear term, using that $a_\varepsilon\gg \varepsilon$, we have
\begin{eqnarray*}
&&\left|\sqrt{2}g \varepsilon \int_{\omega\times Y}\left|\mathbb{D}_{x^\prime}\left[\tilde v\right]+
\frac{1}{a_\varepsilon} \mathbb{D}_{y^\prime}\left[\tilde v\right]+\frac{1}{\varepsilon}\partial_{y_3}
\left[\tilde v\right]\right|dx'dy-\sqrt{2}g \int_{\omega\times Y}\left| 
\partial_{y_3}\left[\tilde {v}'\right]\right|dx'dy\right|\\
&&\leq \sqrt{2}g\int_{\omega\times Y}\left|\varepsilon  \mathbb{D}_{x^\prime}
\left[\tilde v\right]+ \frac{\varepsilon}{a_\varepsilon}\mathbb{D}_{y^\prime}\left[\tilde v\right]+
\partial_{y_3}\left[\tilde v\right]-\partial_{y_3}\left[\tilde v\right]\right|dx'dy\\
&&\leq \sqrt{2}g\int_{\omega\times Y}\left|\varepsilon  \mathbb{D}_{x^\prime}\left[\tilde v\right] \right|dx'dy+
\sqrt{2}g\frac{\varepsilon}{a_\varepsilon}\int_{\omega\times Y}\left| \mathbb{D}_{y^\prime}\left[\tilde v\right]\right|dx'dy \to 0, \text{ as } \varepsilon\to 0.
\end{eqnarray*}
Now, in order to pass the limit in the second nonlinear term, taking into account that
\begin{eqnarray*}
\sqrt{2}g{1 \over \varepsilon}\int_{\omega\times Y}\left|\frac{1}{a_\varepsilon} \mathbb{D}_{y^\prime}\left[\hat u_\varepsilon\right]+
\frac{1}{\varepsilon}\partial_{y_3}\left[\hat u_\varepsilon\right]\right|dx'dy=\sqrt{2}g \int_{\omega\times Y}\left|{\varepsilon \over a_\varepsilon}{1 \over \varepsilon^2} \mathbb{D}_{y^\prime}\left[\hat u_\varepsilon\right]+\frac{1}{\varepsilon^2}\partial_{y_3}\left[\hat u_\varepsilon\right]\right|dx'dy,
\end{eqnarray*}
and using (\ref{convUgorro3}) and the fact that the function $E(\varphi)=|\varphi|$ is proper convex continuous and 
$a_\varepsilon\gg \varepsilon$, we can deduce that
\begin{eqnarray}\label{limit4_case3}
\liminf_{\varepsilon \to 0}\sqrt{2}g{1 \over \varepsilon}\int_{\omega\times Y}\left|\frac{1}{a_\varepsilon} \mathbb{D}_{y^\prime}\left[\hat u_\varepsilon\right]+\frac{1}{\varepsilon}\partial_{y_3}\left[\hat u_\varepsilon\right]\right|dx'dy\ge \sqrt{2}g\int_{\omega\times Y}\left|\partial_{y_3}\left[\hat u\right]\right|dx'dy.
\end{eqnarray}
Moreover, using (\ref{convUgorro3}) the two first terms in the right hand side of (\ref{formvarcvuprime3}) tend to the following limit
\begin{equation}\label{limit5_case3}
\int_{\omega\times Y}f'\cdot (\tilde v'-\hat u')\,dx^\prime dy.
\end{equation}
We consider now the terms which involve the pressure. Taking into account the convergence of the pressure (\ref{convUgorro3}) the first term of the pressure tends to the following limit
$
\int_{\omega\times Y} \hat{P}\, {\rm div}_{x^\prime} \tilde v^\prime\,dx^{\prime}dy,
$
and using (\ref{divmacro3}) and taking into account that $\hat P$ does not depend on $y$, we have (\ref{limit6}). Finally using that ${\rm div}_{y'}\tilde v'=0$, we have (\ref{limit7_case2}).
Therefore, taking into account (\ref{limit6}), (\ref{limit7_case2}) and (\ref{limit1_case3})-(\ref{limit5_case3}), we get
 (\ref{limit_supercritical}).
\end{proof}

\section{Conclusions}

By using dimension reduction and homogenization techniques, we studied the limiting 
behavior of the velocity and of the pressure for a nonlinear viscoplastic Bingham flow with small
yield stress, in a thin porous medium of small height $\varepsilon$  and for which the relative
dimension of the pores is $a_\varepsilon$. Three cases are studied following the value of $\lambda=
\lim_{\varepsilon \rightarrow 0} {a_\varepsilon}/{\varepsilon}$   and, at the limit,  they all 
preserve the nonlinear character of the flow. More precisely, according to \cite{Lions_SanchezPalencia}, 
each of the limit problems \eqref{limit_critical}, \eqref {limit_subcritical} and \eqref{limit_supercritical},
   is written as  a nonlinear Darcy equation:
\begin{equation}\label{Darcy}
\left\{
\begin{array}{ll}
\tilde{U'}(x')=K^{\lambda}\left(f'(x')-\nabla_{x'}\hat{P}(x')\right) & \textrm{ in }  \omega,\\
{\textrm{div}}_{x'}\tilde{U'}(x')=0 &  \textrm{ in } \omega, \\ 
\tilde{U'}(x')\cdot n=0 &  \textrm{ on }  \partial \omega.\\
\end{array}
\right.
\end{equation}
The velocity of filtration $\tilde{U}(x')= \left(\tilde{U'}(x'), \tilde{U}_3 (x') \right)  $ is defined by
\begin{equation*}
\tilde{U}(x')=\int \limits_{Y}\hat{u}(x',y)dy=\int \limits_{0}^{1} \left( \int \limits_{Y'}\hat{u}(x',y', y_3)dy' \right) dy_{3}=
\int \limits_{0}^{1} \tilde{u}(x',y_3) dy_{3}.
\end{equation*}

We remark that in all three cases, the vertical  component $\tilde{U}_3$ of the velocity of filtration equals
zero and this result is in accordance with the previous mathematical studies  of the flow  in this
thin porous medium,  for newtonian fluids (Stokes and  Navier-Stokes equations)
and for power law fluids (see \cite{Fabricius}, \cite{Anguiano}, \cite{Anguiano2}, \cite{Anguiano_SuarezGrau}, 
\cite{Anguiano_SuarezGrau2}). Moreover, despite the fact that the limit pressure is not unique, the velocity 
of filtration   is uniquely determined (see Section 4.3 in \cite{Lions_SanchezPalencia}).
In \eqref{Darcy}, the function $K^{\lambda}:{\mathbb R}^2 \longrightarrow {\mathbb R}^2$ is  nonlinear 
and its expression can not be made explicit for the Bingham flow (see \cite{Lions_SanchezPalencia}).
Nevertheless, in each case, for a given $\xi \in {\mathbb R}^2$, one has
$K^{\lambda}(\xi)=\int \limits_{Y}\chi_{\lambda}^{\xi}(y)dy$,
with $\chi_{\lambda}^{\xi}$ solution of a local problem stated in the cell $Y$. If $0<\lambda<+\infty$, the local problem is 
a 3-D Bingham problem. If $\lambda=0$,  the local problem is  a 2-D
Bingham problem (defined for each  $y_3 \in]0,1[)$, while if $\lambda = +\infty$ the 1-D local problem 
(defined for each $y' \in Y'$)
corresponds to a lower-dimensional Bingham-like law (see \cite{Bunoiu_Kesavan}).

We end with the remark that if in the initial problem \eqref{VA_pressure} we take $g=0$, then the  problem under study 
becomes the Stokes problem. We refer to \cite{Anguiano_SuarezGrau} (case $p=2$) for the asymptotic analysis of the Stokes problem.  
If  we set  $g=0$  in the limit problems \eqref{limit_critical}, \eqref {limit_subcritical} and \eqref{limit_supercritical},
they  become exactly  the ones in  \cite{Anguiano_SuarezGrau},  Theorem 6.1 (case $p=2$), corresponding to the Stokes case. \\

{\bf Acknowledgments:}
Mar\'ia Anguiano has been supported by Junta de Andaluc\'ia (Spain), Proyecto de Excelencia P12-FQM-2466.

\end{document}